\documentclass[10pt, reqno]{amsart}
\usepackage{amssymb}
\usepackage{amsmath}
\usepackage{amscd}
\usepackage{mathrsfs}
\usepackage{float}
\usepackage{graphicx}
\usepackage{amsfonts}
\usepackage{color}
\usepackage{pb-diagram}
\usepackage[utf8]{inputenc}
\usepackage[T1]{fontenc}
\usepackage{framed}
\usepackage{cancel}
\usepackage{hyperref}
\usepackage{times}

\newcommand{\B}{\mathtt{B}}
\newcommand{\Y}{{\rm Y}_{d,n}^{\mathtt B}}

\newcommand{\s}{\mathtt s}
\newcommand{\R}{\mathtt r}

\newcommand{\U}{\mathsf{u}}
\newcommand{\V}{\mathsf{v}}

\newcommand{\x}{\mathsf{x}}
\newcommand{\y}{\mathsf{y}}
\newcommand{\z}{\mathsf{z}}
\newcommand{\m}{\mathtt{m}}
\newcommand{\W}{\mathsf{w}}
\newcommand{\E}{\mathcal{E}_n^{\mathtt{B}}}
\newcommand{\p}{\mathsf{P}(\mathbf{n})}
\newcommand{\pcero}{\mathsf{P}(\mathbf{n}_0)}
\newcommand{\n}{\mathbf{n}}
\newcommand{\w}{{\mathtt m}_1\cdots {\mathtt m}_{n-1}}
\newcommand{\wmenos}{{\mathtt m}_1\cdots {\mathtt m}_{n-2}}
\newcommand{\vnor}{v_i^r\otimes v_j^s}
\newcommand{\vinv}{v_j^s\otimes v_i^r}
\newcommand{\vdimn}{v_{i_1}^{r_1}\otimes \dots \otimes v_{i_n}^{r_n}}

\newcommand{\tmenosn}{\mathbb{T}_{n,k}^{-}}

\newcommand{\tmenos}{\mathbb{T}_{n-1,k}^{-}}

\newtheorem{theorem}{Theorem}
\newtheorem{corollary}{Corollary}
\newtheorem{lemma}{Lemma}
\newtheorem{proposition}{Proposition}
\newtheorem{definition}{Definition}
\newtheorem{remark}{Remark}

\begin{document}
\title{A bt--algebra of type $\mathtt{B}$}

\author{M. Flores}
\thanks{This project was partially supported by PAI 79140019}
\address{Instituto de Matem\'{a}ticas \\ Universidad de Valpara\'{i}so \\ Gran Breta\~{n}a 1091, Valpara\'{i}so, Chile}
\email{marcelo.flores@uv.cl}

%\author{M. Flores}
%\address{Instituto de Matem\'{a}ticas \\ Universidad de Valpara\'{i}so \\ Gran Breta\~{n}a 1091, Valpara\'{i}so, Chile}
%\email{marcelo.flores@uv.cl}

%\author{D. Goundaroulis}
%\address{Center Int\'egratif de G\'enomique,
%UNIL,
%Batiment G\'enopode, CH-1015 Lausanne, Switzerland.}
%\email{dimoklis.gkountaroulis@unil.ch}

\maketitle

\begin{abstract}
  We introduce a bt--algebra of type $\mathtt{B}$. As the original construction of the bt--algebra, we define this bt--algebra of type $\mathtt{B}$ by building from $\Y$. Notably we find a basis for it, a faithful tensorial representation, and we prove that it supports a Markov trace, from which we derive invariants of classical links in the solid torus.
\end{abstract}

\section{Introduction}
The algebra of braids an ties, known as well as the bt--algebra, was defined originally by Aicardi and Juyumaya in \cite{aiju}, having as goal to construct new representations of the braid group. Later, it was observed that its generators and relations have a diagrammatical interpretation in terms of braids and ties, hence its name, see \cite[Section 6]{aijuMMJ}. For a positive integer $n$, the bt--algebra with parameter $\U$ is denoted by $\mathcal{E}_n(\U)$, and its definition is obtained by considering abstractly as the subalgebra of the Yokonuma--Hecke algebra ${\rm Y}_{d,n}:={\rm Y}_{d,n}(\U)$ generated by the braid generators and the family of idempotents that appear in the quadratic relations of these generators. Thus, there is a natural homomorphism from $\mathcal{E}_n$ in ${\rm Y}_{d,n}$, which is injective for $d\geq n$, see \cite{esry}, cf. \cite[Remark 3]{aijuMMJ}.\smallbreak
In \cite{aiju} was proved that the bt--algebra has finite dimension, although that the autors couldn't get an explicit basis for it. This problem was solved later by Ryom-Hansen in \cite{rh}, who constructed a basis for $\mathcal{E}_n$, proving that the dimension of the algebra is $b_n n!$, where $b_n$ is the n-th bell number. He also constructed a faithful tensorial representation (Jimbo--type) of this algebra, which is used to classify the irreducible representations of $\mathcal{E}_n$, and plays a essential role in the proof of the linear independency of the basis proposed there.
\smallbreak
Later, in \cite{aijuMMJ} is proved that the algebra $\mathcal{E}_n$ supports a Markov trace, this is achieved by using the method of relative traces, and the basis provided by Ryom-Hansen; for more examples of the relative traces method see \cite{chpoIMRN}, \cite{fjl}, \cite{isog}. Then, by using this trace as ingredient in the Jones's recipe \cite{jo}, they define an invariant $\overline{\Delta}(\U,\mathsf{A},\mathsf{B})$ for classical knots (respectively $\overline{\Gamma}(\U,\mathsf{A},\mathsf{B})$ for singular knots) with parameters $\U, \mathsf{A}$ and $\mathsf{B}$. It's worth to say that for links, the invariant $\overline{\Delta}$ is more powerful than the Homeflypt polynomial, see \cite[Addendum]{aijuMMJ}.
\smallbreak
The bt--algebra have been studied by several researchers lately, which has helped to respond some important questions of its structure. In \cite{esry} J. Espinoza and Ryom-Hansen found a cellular basis for the bt-algebra, which is used to obtain an isomorphism theorem between $\mathcal{E}_n$ and a sum of matrix algebras over certain wreath products. In her Ph.D. tesis \cite{Ba} E. Banjo got an explicit isomorphism between the specialization $\mathcal{E}_n(1)$ and the small ramified partition algebra \cite{PMa}, and using it, she determined the complex generic representation of $\mathcal{E}_n$. Finally, in \cite{Ma} I. Marin introduced a generalization of the bt--algebra, more precisely, given any Coxeter system $(W,S)$, he defined an extension of the corresponding Iwahori--Hecke algebra, denoted by $C_W$, which coincide with the algebra $\mathcal{E}_n$ when $W$ is the Coxeter group of type $\mathtt{A}$.
\smallbreak
Recently, in \cite{fjl} we introduce a framization of the Hecke algebra of type $\mathtt{B}$, denoted by $\Y$, which is some kind of analogous of the Yokonuma--Hecke algebra for the $\mathtt{B}$-type case, hence its notation. As we recalled above, the definition of the algebra $\mathcal{E}_n(\U)$ is strongly related with certain subalgebra of ${\rm Y}_{d,n}$. Then, it is natural try to define an analogue of the bt--algebra, this time, as a subalgebra of $\Y$. Thus, in this paper we introduce a new algebra, denoted by $\E=\E(\U)$, that contains the algebra of braids and ties, and we can say that it is a bt-algebra of type $\mathtt{B}$. Moreover, we also construct a basis and a tensorial representation for it (adjusting the ideas given by Ryom-Hansen in \cite{rh}), having as goal to prove that $\E$ supports a Markov trace, which is the main result of this work. It is important to note that the algebra introduced here doesn't coincide with the given by I. Marin considering $W$ as the Coxeter group of $\mathtt{B}$ type, see Remark~\ref{Marin}.

\smallbreak
The article is organized as follows. In Section 2 we fix some notations and recall the results used in the paper. In Section 3, making an analogy with the classical case, we introduce the algebra $\E$, which contains the bt--algebra. Also, we give some relations that hold on it (which are, mostly, direct consequence of results from \cite{fjl} and \cite{rh}), and we propose a diagrammatical interpretation for $\mathcal{E}_n$, in the sense of \cite[Section 6]{aijuMMJ}. In Section 4 we construct two linear bases for $\E$ readjusting the ideas given in \cite{rh}. Similarly as in \cite{fjl}, one of these bases has a technical role, and the other is used for define a Markov trace in the last section. Moreover, we also construct a faithful tensorial representation of $\E$, which is the natural extension of the representation of the bt--algebra given by Ryom-Hansen, and plays a key role in the proof of the linear independency of one of the bases given here, see Theorem~\ref{basis} (cf. \cite[Theorem 3]{rh}). In Section 5 we prove that $\E$ supports a Markov trace (Theorem~\ref{Markovtrace}), we prove that, constructing a family of relative traces using the basis given in the previous section. We use this method since as in the classical case, the basis obtained here cannot be defined in an inductive manner, then it is extremely difficult to define a Markov trace analogously to the Ocneanu's trace \cite{jo}. Thus, keeping the approach in \cite{aijuMMJ}, we split the proof of the main result in several lemmas, which give step by step the necessary conditions for the trace. Finally, using our trace as ingredient in Jones's recipe we define an invariant of classical links in the solid torus, which restricted to classical links (that is, braids of $\mathtt{B}$--type without the \lq loop generator\rq involved, see Section 2) coincide with $\overline{\Delta}$, and therefore it is more powerful than the Homflypt polynomial, whenever is evaluated in classical links.\\

\noindent \textbf{Acknowledgements.} The results contained in this manuscript were obtained during a research visit to the Maths Section of ICTP, in Trieste, Italy, which I thank its support and hospitality. A particular thank also to F. Aicardi, for her comments about the diagrammatical interpretation of the algebra introduced here, which were essential to develop this work.

\section{Preliminaries}
In this section we review known results, necessary for the sequel, and we also fix the following terminology and notations that will be used  along the article:
\begin{itemize}

\item[--] The letters $\U,\V$ denote indeterminates. Consider ${\Bbb K}:={\Bbb C}(\U,\V)$.

\item[--] The term {\it algebra} means unital associative algebra over  ${\Bbb K}$.

\item[--] The sets $\{0,1,\dots,n\}$ and $\{1,\dots,n\}$ will be denoted simply by $\n_0$ and $\n$ respectively.

\item[--] As usual, we denote by $\ell$ the length function associated to the Coxeter groups.
\end{itemize}

\subsection{}\label{typeB} Set $n\geq 1$. Let us denote by $W_n$ the Coxeter group of type ${\mathtt  B}_n$. This is the finite Coxeter group
 associated to the following Dynkin diagram
\begin{center}
\setlength\unitlength{0.2ex}
\begin{picture}(350,40)
\put(82,20){$\R_1$}
\put(120,20){$\s_{1}$}
\put(200,20){$\s_{n-2}$}
\put(240,20){$\s_{n-1}$}

\put(85,10){\circle{5}}
\put(87.5,11){\line(1,0){35}}
\put(87.5,9){\line(1,0){35}}
\put(125,10){\circle{5}}
\put(127.5,10){\line(1,0){10}}

\put(145,10){\circle*{2}}
\put(165,10){\circle*{2}}
\put(185,10){\circle*{2}}

\put(205,10){\circle{5}}
\put(207.5,10){\line(1,0){35}}
\put(245,10){\circle{5}}
\put(192.5,10){\line(1,0){10}}

%\put(100,7){$<$}

\end{picture}
\end{center}

Define $\R_k=\s_{k-1}\ldots \s_1 \R_1 \s_1\ldots \s_{k-1}$ for $2\leq k\leq n$. It is known, see \cite{gela},  that every element $w\in W_n$ can be written uniquely as $w=w_1\ldots w_n$ with $w_k\in \mathtt{N}_k$, $1\leq k\leq n$, where
\begin{equation}\label{NWn}
\mathtt{N}_k:=\left\{
1, \R_{k},
\s_{k-1}\cdots \s_{i},
\s_{k-1}\cdots \s_{i}\R_{i}\, ;\, 1\leq i \leq k-1
\right\}.
\end{equation}
Moreover, this expression for $w$ is reduced. Hence, we have $\ell(w)=\ell(w_1)+\cdots +\ell(w_n)$.\\

Further, the group $W_n$ can be realized as a subgroup of the permutation group of the set  $X_n:=\{-n, \ldots , -2, -1, 1, 2, \ldots, n\}$. Specifically, the elements of $W_n$ are the permutations $w$ such  that $w(-m) = - w(m)$, for all $m \in X_n$. Then, the elements of $W_n$ can be parameterized  by the elements of $X_n^n:=\{(m_1,\dots,m_n)\ |\ m_i\in X_n\ \text{for all $i$}\}$ (see \cite[Lemma 1.2.1]{gr}). More precisely, the element $w\in W_n$ corresponds to the element $(m_1, \ldots ,m_n)\in X_n^n$ such that $m_i= w(i)$, for details see \cite[Section 1.3]{fjl}.\\

The corresponding \textit{braid group of type} ${\mathtt  B}_n$ associated to $W_n$, is defined as the group $\widetilde{W}_n$ generated  by $\rho_1 , \sigma_1 ,\ldots ,\sigma_{n-1}$ subject to the following relations
 \begin{equation}\label{braidB}
\begin{array}{rcll}
 \sigma_i \sigma_j & =  & \sigma_j \sigma_i & \text{ for} \quad \vert i-j\vert >1,\\
  \sigma_i \sigma_j \sigma_i & = & \sigma_j \sigma_i \sigma_j & \text{ for} \quad \vert i-j\vert = 1,\\
   \rho_1\sigma_i&=&\sigma_i\rho_1 &\text{ for}\quad i>1,\\
\rho_1 \sigma_1 \rho_1\sigma_1 & = & \sigma_1 \rho_1 \sigma_1\rho_1. &
  \end{array}
\end{equation}

Geometrically, braids of type ${\mathtt  B}_n$ can be viewed as classical braids of type ${\mathtt  A}_{n}$ with $n+1$ strands, such that the first strand is identically fixed. This is called `the fixed strand'. The 2nd, \ldots, $(n+1)$st strands are renamed from 1 to $n$ and they are called `the moving strands'. The `loop' generator $\rho_1$ stands for the looping of the first moving strand around the fixed strand in the right-handed sense, see \cite{la1,la2}. In Figure \ref{bbraid} we illustrate a braid of type ${\mathtt  B}_4$.
\begin{figure}[h]
\begin{center}
  \includegraphics{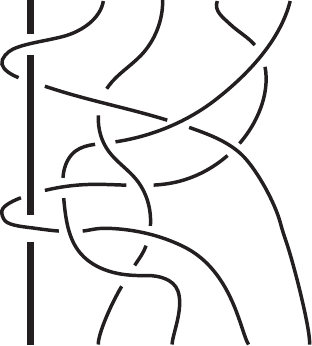}
	\caption{A braid of ${\mathtt  B}_4$-type.}\label{bbraid}
	\end{center}
 \end{figure}
\subsection{} Recently in \cite{fjl} we introduced a new framization of the Hecke algebra of type $\mathtt B$, denoted by $\Y:=\Y(\U,\V)$. This algebra was constructed searching an analogous of the Yokonuma--Hecke algebra for the type $\mathtt{B}$ case, this, with the final objective to explore their usefulness in knot theory. Thus, in this article we constructed two linear bases, a faithful tensorial representation of Jimbo type for ${\rm Y}_{d,n}^{\mathtt{B}}(\U,\V)$,  and we proved that $\Y$ supports a Markov trace. Finally we defined, by using Jones's recipe, a new invariant for framed and classical links in the solid torus. Along this paper we use several properties of this algebra, then we recall some of them. We begin with its definition

\begin{definition}\rm
Let $n\geq 2$. The algebra ${\rm Y}_{d, n}^{\mathtt{B}} := {\rm Y}_{d, n}^{\mathtt{B}}(\U, \V)$, is defined as the algebra over $\mathbb{K}:=\mathbb{C}(\U,\V)$ generated by framing generators $t_1,\dots,t_n$, braiding generators $g_1,\dots,g_{n-1}$ and the loop generator $b_1$, subject to the following relations

\begin{eqnarray}g_ig_j & = & g_jg_i  \quad \text{ for} \quad \vert i-j\vert > 1,\label{braid1}\\
  g_i g_j g_i & = & g_j g_i g_j \quad \text{ for} \quad \vert i-j\vert = 1, \label{braid2}\\
  b_1 g_i & = & g_i b_1 \quad \text{for all}\quad  i\not= 1, \label{braid3}\\
b_1 g_1 b_1 g_1 & = &  g_1 b_1 g_1 b_1, \label{braid4}\\
t_i t_j & =  & t_j t_i  \quad \text{for all }  \  i, j,\label{modular2}\\
  t_j g_i & =  & g_i t_{s_i(j)}\quad \text{for all }\,  i, j,  \label{th}\\
t_i b_1 & = & b_1 t_i \quad \text{for all i},\quad   \label{tb1}\\
t_i^d & = & 1 \quad \text{for all}\quad i, \label{modular1}\\
g_i^2 & = &  1+ (\U-\U^{-1})e_ig_i\quad \text{for all $i$}, \label{quadraticU}\\
b_1^2 & = &  1 + (\V-\V^{-1})f_1b_1.\label{quadraticV},
\end{eqnarray}
where
 $$
e_i:=\frac{1}{d}\sum_{s=0}^{d-1} t_i^st_{i+1}^{d-s} \quad \text{and}\quad f_j:=\frac{1}{d}\sum_{s=0}^{d-1} t_j^s;\quad \text{for $1\leq i\leq n-1$, and $1\leq j\leq n.$}
$$
For n=1, we define the algebra ${\rm Y}_{d,1}^{\B}$ as the algebra generated by $1, b_1$ and $t_1$ satisfying the relations (\ref{tb1}), (\ref{modular1}) and (\ref{quadraticV}).

\end{definition}
Notice that the elements $f_j$'s and $e_i$'s are idempotents. Also, It is clear that the element $f_1$ commutes with $b_1$  and  $e_i$ commutes with $g_i$. These facts imply that the generators $b_1$ and the $g_i$'s are invertible. Namely, we have:
\begin{equation}\label{inverse}
b_1^{-1} = b_1 - (\V-\V^{-1})f_1\quad \text{and}\quad g_i^{-1} = g_i - (\U -\U^{-1})e_i.
\end{equation}

Now we recall the two basis of $\Y$ given in \cite{fjl}, which we will useful for the sequel.\smallbreak
Set $\overline{b}_1 := b_1$, $\overline{b}_k := g_{k-1}\ldots g_1 b_1 g_1\ldots g_{k-1}$, and $b_k := g_{k-1}\ldots g_1 b_1 g_1^{-1}\ldots g_{k-1}^{-1}$ for all $2\leq k\leq n$. For all $1\leq k\leq n$, let us define inductively  the sets $N_{d,k}$ by
$
N_{d,1} := \{t_1^m, \overline{b}_1 t_1^m \,;\, 0 \leq m \leq d-1\}
$
and
$$
N_{d,k} := \{t_k^m, \overline{b}_{k}t_k^m, g_{k-1} x\,;\, x \in N_{d,k-1},\, 0\leq m\leq d-1\}
\quad \mbox{for all $2 \leq k\leq n$.}
$$

Analogously, for all $1\leq k\leq n$ we define inductively the sets $M_{d,k}$ exactly like $N_{d,k}$'s but exchanging $\overline{b}_k$ by $b_k$ in every case. \\

 Finally, consider $\mathsf{D}_n=\{ \mathfrak{n}_1\mathfrak{n}_2\cdots \mathfrak{n}_n\ |\mathfrak{n}_i\in N_{d,i} \}$ and $\mathsf{C}_n=\{ \mathfrak{m}_1\mathfrak{m}_2\cdots \mathfrak{m}_n\ |\mathfrak{m}_i\in M_{d,i} \}$. Then, we have that $\mathsf{D}_n$ and $\mathsf{C}_n$ are bases of $\Y$, for details see \cite[Section 4]{fjl}.

\subsection{} We denote by $\p$ the set formed by the set-partitions of $\mathbf{n}$, recall that the cardinality of $\p$ is the $n$--th Bell number, denoted by $b_n$.
The subsets of $\n$ entering a partition are called blocks. For short we shall omit the subset of cardinality 1 (single blocks) in the partition. For example, the partition $I=(\{1,2,3\},\{4,6\},\{5\},\{7\})$ in $\mathsf{P}(\mathbf{7})$, will be simply written as $I=(\{1,2,3\},\{4,6\})$. Moreover, Supp(I) will be denote the union of non--single blocks  of $I$.\smallbreak
The symmetric group $S_n$ acts naturally on $\p$. More precisely, set $I=(I_1,\dots,I_m)\in \p$ . The action $w(I)$ for a $w\in S_n$ is given by
$$w(I)=((w(I_1),\dots,w(I_m)))$$
where $w(I_k)$ is the set obtained of applying $w$ to the set $I_k$.\smallbreak
The pair $(\p, \preceq)$ is a poset. Specifically, given $I=(I_1,\dots,I_m),\ J=(J_1,\dots,J_s) \in \p$, the partial order $\preceq$ is defined by.
$$I\preceq J\qquad \text{if and only $J$ is a union of some blocks of $I$}$$
When $I\preceq J$, we will say that $I$ refines $J$.\smallbreak
Let $I, J \in \p$, we denote $I*J$ the minimal set partition refined by $I$ and $J$. Let $A$ a subset of $\n$. Along the work we will use for short $I*A$ instead $I*(A)$. Thus, if $I=(I_{1},\dots,I_{k},I_{i_{k+1}},\dots,I_{i_{m}})$, where the first $k$ blocks are the blocks that have intersection with $A$, and the rest are those that don't have. Then $I*A$ is given by%$I_{j_s}\cap L\not =\emptyset$ for all $1\leq s \leq k$, and $I_{i_r}\cap L =\emptyset$ for all $1\leq r \leq m-k$. We have that $I*L$ is given by

$$I*A=(A',I_{k+1},\dots,I_{m})$$
where $A'=A\cup I_{1}\cup \dots \cup I_{k}$. In particular, $I*\{j,m\}$ coincides with $I$ if $j$ and $m$ already belong to the same block, otherwise, $I*\{j,m\}$ coincides with $I$ except for the blocks containing $j$ and $m$, which merge in a sole block. For short, we will write $I*j$ instead of $I*\{j,j+1\}$. For instace, for the set partition $I=(\{1,4\},\{2,5\},\{3,6,7\})$ in $\mathsf{P}(\mathbf{8})$:
$$I*\{4,5,8\}=(\{1,2,4,5,8\},\{3,6,7\})\quad \text{,} \quad I*2=(\{1,4\},\{2,3,5,6,7\})\quad \text{and} \quad I*6=I$$
Also, for $I\in \p$, we denote $I\backslash n$ the element in $\mathsf{P}(\mathbf{n-1})$ that is obtained by removing $n$ from $I$. Let be $\mathcal{P}(\n_0)$ the set of partitions of $\n_0$, note that, $\mathsf{P}(\n_0)$ is essentially $\mathsf{P}(\mathbf{n+1})$, then all the definitions and notations above are valid for partitions in $\pcero$. Finally, for $A\subseteq \n_0$  we define $A^*=A\backslash \{0\}$.

\section{An algebra of braids and ties inside $\Y$ }

In this section we propose a generalization of the algebra of braids and ties $\mathcal{E}_n(\U)$ defined originally in \cite{aiju} and posteriorly studied in \cite{aijuMMJ,aiju2}. As we note previously, the definition of $\mathcal{E}_n(\U)$ definition was obtained by considering abstractly as a subalgebra of ${\rm Y}_{d,n}(\U)$. In \cite{fjl} we introduce a framization of the Hecke algebra of type $\mathtt{B}$, denoted by $\Y$, which is the some kind of analogous of the Yokonuma--Hecke algebra for the $\mathtt{B}$-type case. Then, it is natural to define an analogue of the bt--algebra, this time, considering a subalgebra of $\Y$, which carries us to the next definition.

\begin{definition}
 Let $n\geq 2$. We define the bt--algebra of type $\B$, denoted by $\E=\E(\U,\V)$, as the algebra generated by $B_1, T_1\dots, T_{n-1}$ and $F_1,\dots F_n, E_1\dots, E_{n-1}$, subject to the following relation
 \begin{eqnarray}
    T_iT_j &=& T_jT_i\qquad \mbox{for all $|i-j|>1$} \label{first}\\
    T_iT_{i+1}T_i &=&T_{i+1}T_iT_{i+1} \qquad \mbox{for all $1\leq i\leq n-2$} \\
    T_i^2&=&1+(\U-\U^{-1})E_iT_i \qquad \mbox{for all $1\leq i\leq n-1$} \\
    E_i^2 &=& E_i  \qquad \mbox{for all $i$}\\
    E_iE_j&=&E_jE_i  \qquad \mbox{for all $i, j$} \\
    E_iT_i&=&T_iE_i   \qquad \mbox{for all $1\leq i\leq n-1$}\\
    E_iT_j&=&T_jE_i \qquad \mbox{for all $|i-j|>1$} \\
    E_iE_jT_i&=&T_iE_iE_j= E_j T_i E_j \qquad \mbox{for all $1\leq i,j\leq n-1$ such that $|i-j|=1$} \\
    E_iT_jT_i&=&T_jT_iE_j\qquad \mbox{for all $1\leq i,j\leq n-1$ such that $|i-j|=1$}\label{lastA}\\
    B_1T_1B_1T_1&=&T_1B_1T_1B_1  \\
    B_1T_i&=&T_iB_1 \qquad \mbox{for all $i>1$}\\
    B_1^2&=&1+(\V-\V^{-1})F_1B_1\label{quadratictipoB}\\
    B_1E_i&=&E_iB_1\qquad \mbox{for all $i$} \label{last} \\
    F_i^2&=&F_i \qquad \mbox{for all $i$} \label{idempotentf} \\
    B_1F_j&=&F_jB_1 \qquad \mbox{for all $j$} \label{actb}\\
    F_iE_j&=&E_jF_i \qquad \mbox{for all $i,j$}\\
    F_jT_i&=&T_iF_{s_i(j)} \label{perm}\qquad \text{where $s_i$ is the transposition $(i,i+1)$}\\
    E_iF_i&=&F_iF_{i+1}=E_iF_{i+1} \label{ef}\qquad \mbox{for all $1\leq i\leq n-1$}
 \end{eqnarray}
\end{definition}
For $n=1$ we define the algebra $\mathcal{E}_1^{\B}$ as the algebra generated by $1, B_1$ and $F_1$ subject to the relations (\ref{quadratictipoB}),(\ref{idempotentf}) and (\ref{actb}).\\

  More precisely, the definition of $\E$ is obtained by considering abstractly the subalgebra of $\Y$ generated by $b_1$, and the elements $g_i$'s, $e_i$'s and $f_i$'s. Thus, considering $g_i$ as $T_i$, $e_i$ as $E_i$, $f_i$ as $F_i$ and $b_1$ as $B_1$, the defining relations of $\E$ correspond to the set of relations derived from the relations (\ref{braid1})-(\ref{quadraticV}) of $\Y$.
\begin{proposition}\label{homoEY}
  There is a natural homomorphism $\varphi_n:\E\rightarrow \Y$ defined by the mapping $T_i\mapsto g_i$, $E_i\mapsto e_i$, $F_i\mapsto f_i$ and $B_1\mapsto b_1$.
\end{proposition}

\begin{remark}\label{Marin} \rm
Let be $(W,S)$ a Coxeter system, and $C_W$ the algebra introduced by I. Marin in \cite{Ma}. It is known that $C_W=\mathcal{E}_n$ when $W$ is the Coxeter group of type $\mathtt{A}_{n-1}$, then, it is natural to think that $C_W$ should coincide with $\E$ when $W=W_n$, but this doesn't happen. In fact, we will prove in Section 4 that the dimension of $\E$ is $b_{n+1}|W_n|$, meanwhile $C_W$ has dimension $Bell(W)|W|$, where the $Bell(W)$ is a entire number called the Bell number of $W$ (by obvious reasons). These numbers are not completely determined for type $\mathtt{B}_n$, but are known for low dimensions, more precisely, for $n\geq 2$ the sequence of dimensions is the following: 8, 38, 218, 1430, 10514,\dots; for details see \cite[Section 3.6]{Ma}. Thus, we have that these algebras are different, which indicates that the algebra $\E$ should be interesting by itself.
\end{remark}
\begin{remark}\label{changvar}\rm
  Note that the quadratic relation for ${\rm Y}_{d,n}$ and $\mathcal{E}_n$ used in \cite{aiju, aijuMMJ} is different than the used here. However, it is well known that modifying the set of generators of ${\rm Y}_{d,n}$ (respectively $\mathcal{E}_n$), it is possible to get a presentation with the desired quadratic relation . More precisely, if $\widetilde{g}_i$ (respectively $\widetilde{T}_i$) denote the original braid generators of the Yokonuma--Hecke algebra (respectively, of the bt--algebra), and $u$ the parameter used in the usual quadratic relation. Then, by taking $u=\U^2$ and $g_i=\widetilde{g}_i+(\U^{-1}-1)e_i\widetilde{g}_i$ (respectively $T_i=\widetilde{T}_i+(\U^{-1}-1)e_i\widetilde{T}_i$), we obtain a presentation including the quadratic relation used by us, cf. \cite[Remark 1]{chporep}. Consequently, the bt-algebra $\mathcal{E}_n$ can be regarded as the algebra generated by the elements $T_i$'s and $E_i$'s subject to the relations (\ref{first})-(\ref{lastA}), thus, in particular, we have that $\mathcal{E}_n\subseteq \E$.
\end{remark}

\begin{proposition}
   The map $\varphi:\E\rightarrow \mathcal{E}_n$ given by $\varphi(T_i)=T_i$, $\varphi(E_i)=E_i$, $\varphi(B_1)=1$ and $\varphi(F_i)=1$ define a natural epimorphism.
\end{proposition}

 %In fact we can present $\E$ as the algebra generated by $B_1, T_1\dots, T_{n-1}$ and $F_1, E_1\dots, E_{n-1}$
%\begin{eqnarray*}
%   F_1^2&=&F_1 \qquad \mbox{for all $i$} \\
%    B_1F_1&=&F_1B_1 \qquad \mbox{for all $j$} \\
%    F_1E_j&=&E_jF_1 \qquad \mbox{for all $i,j$}\\
%    F_1T_1&=&T_1F_i(j)} \label{perm}\qquad \text{where $s_i$ is the transposition $(i,i+1)$}\\
%    E_iF_i&=&F_iF_{i+1}=E_iF_{i+1} \qquad \mbox{for all $1\leq i\leq n-1$}
%\end{eqnarray*}

%{\color{red}Thus, as in  \cite{aiju2} we can see the generators of the algebra geometrically as follows.
%\begin{figure}[h]
%  \centering
%  \includegraphics{generadoresbt.pdf}\label{generators}
%\end{figure} (Maybe extend a little bit)}

Now, we recall some useful result from  \cite{rh}. For $i<j$, we define $E_{i,j}$ by
$$E_{i,j}=\left\{\begin{array}{ll}
           E_i  & \text{for $j=i+1$}  \\
           T_i\dots T_{j-2}E_{j-1}T_{j-2}^{-1}\dots T_i^{-1}  &  \text{otherwise}
          \end{array}\right.
$$
For a nonempty subset $J$ of $\mathbf{n}$ we define $E_J=1$ if $|J|=1$ and
$$E_J:=\prod_{(i,j)\in J\times J,\ i<j}E_{i,j}$$
Note that $E_{\{i,j\}}=E_{i,j}$. Also we have by  \cite[Lemma 4]{rh} that
$$E_j=\prod_{j\in J, j\not=i_0}E_{i_0,j},\quad \text{where $i_0={\rm min}(J)$}$$
Moreover, for $I=(I_1,\dots, I_m)\in \mathsf{P}(\mathbf{n})$ we define $E_I$ by
$$E_I=\prod_{k}E_{I_k}.$$
In the same fashion, for any subset $A$ of $\mathbf{n}$ we define
$$F_A=\prod_{i\in A}F_i$$
\smallbreak

\begin{remark}\rm
Notice that by (\ref{perm}) we have that
\begin{equation}\label{fi}
  F_i=T_{i-1}\dots T_1 F_1 T_1^{-1}\dots T_{i-1}^{-1}=T_{i-1}^{-1}\dots T_1^{-1} F_1 T_1\dots T_{i-1}.
\end{equation}

Then, we can omit $F_2,\dots, F_n$ from the presentation (changing/adding certain relations simultaneously!), but we prefer include them, since the computations become simpler by using them. Also note that, using (\ref{perm}), we can obtain a generalization of (\ref{ef}) by conjugating it. More precisely we have
 \begin{equation}\label{efgen}
   E_{i,j}F_i=E_{i,j}F_j=F_iF_j\qquad \text{for all $1\leq i<j\leq n$}
 \end{equation}
\end{remark}

Now, we introduce certain elements that will be used for the construction of a linear basis of $\E$. Let be $\overline{B}_1=B_1$, and for $k\geq 2$ we define
$$\overline{B}_k:=T_{k-1}\dots T_1 B_1 T_1\dots T_{k-1}$$
$$B_k:=T_{k-1}\dots T_1 B_1 T_1^{-1}\dots T_{k-1}^{-1}$$

Further, we considere the sets $\mathsf{M}_{k}$ defined inductively by
$
\mathsf{M}_{1} := \{1, B_1 \}
$
and
$$
\mathsf{M}_{k} := \{1,\ B_k,\ T_{k-1}x\ |\ x\in \mathsf{M}_{k-1}\}
\quad \mbox{for all $2 \leq k\leq n$.}
$$
Analogously, for all $1\leq k\leq n$ we define inductively the sets $\mathsf{N}_{k}$'s exactly like $\mathsf{M}_{k}$'s but exchanging $B_k$ by $\overline{B}_k$ in each case.\smallbreak

Now notice that  every element of $\mathsf{M}_{k}$ has the form $\mathbb{T}_{k,j}^+$ or  $\mathbb{T}_{k,j}^-$ with $j\leq k$,  where
$$
\mathbb{T}_{k,k}^+:=1, \qquad \mathbb{T}_{k,j}^+ := T_{k-1}\cdots T_j\quad \text{for}\ j<k,
$$
and
$$
\mathbb{T}_{k,k}^-:=B_k, \qquad \mathbb{T}_{k,j}^- := T_{k-1}\cdots T_jB_{j}\quad \text{for}\ j<k.
$$

 Similar expressions exist for elements in $\mathsf{N}_{k}$ exchanging $B_k$ by $\overline{B}_k$ as well, which will be denoted by $\mathbb{\overline{T}}_{k,j}^+$ and  $\mathbb{\overline{T}}_{k,j}^-$.

The following results are direct consequence of \cite{fjl}, and these will be used frequently in the sequel

\begin{lemma}\label{relaciones'}For $n\geq 2$ the following relations hold\vspace{2mm}
  \begin{itemize}
   \item[i)]$\overline{\mathbb{T}}^{-}_{n,k}T_j=\left\{\begin{array}{cc}
                                                    T_j\overline{\mathbb{T}}^{-}_{n,k}     & j<k-1 \\
                                                       \overline{\mathbb{T}}^{-}_{n,k-1}+(\U-\U^{-1} )\overline{\mathbb{T}}^{-}_{n,k}E_j   & j=k-1 \\
                                                       \overline{\mathbb{T}}^{-}_{n,k+1}   & j=k \\
                                                     T_{j-1}\overline{\mathbb{T}}^{-}_{n,k}      & j>k
                                                       \end{array}
\right.$\vspace{3mm}
   \item[ii)] $\overline{\mathbb{T}}_{n,k}^{+}T_j=\left\{\begin{array}{cc}
                                                    T_j\overline{\mathbb{T}}^{+}_{n,k}     & j<k-1 \\
                                                 \overline{\mathbb{T}}_{n,k-1}^{+}      & j=k-1 \\
                                        \overline{\mathbb{T}}_{n,k+1}^{+}+(\U-\U^{-1})\overline{\mathbb{T}}_{n,k}^{+}E_j               & j=k \\
                                                     T_{j-1}\overline{\mathbb{T}}^{+}_{n,k}      & j>k
                                                       \end{array}
\right.$ \vspace{3mm}
  \item[iii)] $\overline{\mathbb{T}}^{-}_{n,k} B_1=\left\{\begin{array}{cc}
                                                               \overline{\mathbb{T}}_{n,k}^{+} + (\V-\V^{-1})\overline{\mathbb{T}}_{n,k}^{-}F_1       & \text{for $k=1$} \\
                                                                 B_1\overline{\mathbb{T}}^{-}_{n,k}        &  \text{for $k\not=1$}
                                                                    \end{array}\right.
$\vspace{3mm}
\item[iv)]$\overline{\mathbb{T}}^{+}_{n,k} B_1=\left\{\begin{array}{cc}
                                                               \overline{\mathbb{T}}^{-}_{n,k}      & \text{for $k=1$} \\
                                                                 B_1\overline{\mathbb{T}}^{+}_{n,k}        &  \text{for $k\not=1$}
                                                                    \end{array}\right.$
 \end{itemize}
In particular, we have that
\begin{itemize}
  \item[--] $\overline{B}_nT_j=\left\{\begin{array}{cc}
                                  T_j\overline{B}_n & \text{for $j< n-1$} \\
                                  T_{n-1}\overline{B}_{n-1}+(\U-\U^{-1})\overline{B}_nE_{n-1} & \text{for $j= n-1$}
                                \end{array}\right.
$\vspace{3mm}
  \item[--] $\overline{B}_nT_{n-1}=T_{n-1}\overline{B}_{n-1}+(\U-\U^{-1})\overline{B}_nE_{n-1}$
\end{itemize}

\end{lemma}
\begin{proof}
  See \cite[Lemma 3 and 4]{fjl}.
\end{proof}
\begin{lemma}\label{relaciones}For $n\geq 2$ the following relations hold\vspace{2mm}
  \begin{itemize}
   \item[i)]$\mathbb{T}^{\pm}_{n,k}T_j=\left\{\begin{array}{cc}
                                                    T_j\mathbb{T}^{\pm}_{n,k}     & j<k-1 \\
                                      \mathbb{T}^{\pm}_{n,k-1}                 & j=k-1 \\
                                                \mathbb{T}^{\pm}_{n,k+1}+(\U-\U^{-1})\mathbb{T}^{\pm}_{n,k}E_j       & j=k \\
                                                     T_{j-1}\mathbb{T}^{\pm}_{n,k}      & j>k
                                                       \end{array}
\right.$\vspace{3mm}

  \item[ii)] $\mathbb{T}^{-}_{n,k} B_1=\left\{\begin{array}{cc}
                                                               \mathbb{T}^{+}_{n,k} + (\V-\V^{-1})\mathbb{T}^{-}_{n,k}F_1       & \text{for $k=1$} \\
                                                                        B_1\mathbb{T}^{-}_{n,k}+ (\U-\U^{-1})\alpha_{n,k} &  \text{for $k\not=1$}
                                                                    \end{array}\right.
$\vspace{3mm}
\item[iii)]$\mathbb{T}^{+}_{n,k} B_1=\left\{\begin{array}{cc}
                                                              \mathbb{T}^{-}_{n,k}      & \text{for $k=1$} \\
                                                                 B_1\mathbb{T}^{+}_{n,k}        &  \text{for $k\not=1$}
                                                                    \end{array}\right.$
 \end{itemize}
where $\alpha_{n,k}=\left(B_1T_{1}^{-1}\dots T_{k-2}^{-1} \mathbb{T}_{n,1}^- E_{1,k}-T_{1}^{-1}\dots T_{k-2}^{-1}\mathbb{T}_{n,1}^- B_1E_{1,k}\right)$. In particular, we have\vspace{2mm}
  \begin{itemize}
    \item[--] $B_nT_{n-1}=T_{n-1}B_{n-1}$\vspace{2mm}
    \item[--]$B_nT_j=T_jB_n$, for all $j<n-1$\vspace{2mm}
   \item[--]$B_nB_1=B_1B_n+ (\U-\U^{-1})[B_1T_{1}^{-1}\dots T_{n-2}^{-1} \mathbb{T}_{n,1}^- E_{1,k}-T_{1}^{-1}\dots T_{n-2}^{-1}\mathbb{T}_{n,1}^-B_1E_{1,k}]$\vspace{2mm}
 \end{itemize}

\end{lemma}

\begin{proof}
 See \cite[Lemma 5,6 and 7]{fjl}.
\end{proof}

\begin{lemma}\label{restantes}
The following claims hold in $\E$.

  \begin{itemize}
    \item[(i)] $T_kB_kB_{k+1}=B_kT_kB_k$,\quad for $k\geq 1$.
     \item[(ii)] $\mathbb{T}^{-}_{k,j}B_k=B_{k-1}\mathbb{T}^{-}_{k,j}$,\quad for $k\geq 2$.
  \end{itemize}
  \end{lemma}
  \begin{proof}
    See (iv) \cite[Proposition~2]{fjl} and (i) \cite[Lemma 8]{fjl} respectively.
  \end{proof}

The defining generators $B_1$ and  $T_i$'s of the algebra $\E$ satisfy the same braid relations as the Coxeter generators  $\R_1$ and $\s_i$ of the group $W_n$. Thus, the well--known  Matsumoto's Lemma  implies that  if $w_1\ldots w_m$ is a reduced expression of $w\in W_n$, with $w_i\in \{\R_1, \s_1, \ldots, \s_{n-1}\}$, then the following element $T_w$ is well--defined:
\begin{equation}\label{gw}
T_w := T_{w_1}\cdots  T_{w_m},
\end{equation}
where $T_{w_i} = B_1$, if $w_i=\R_1$ and $T_{w_i} = T_j$, if  $w_i = \s_j$. Therefore, according to \ref{typeB} we have that $\{T_w\ |\ w\in W_n\}=\{\mathfrak{r}_1\dots\mathfrak{r}_n\ |\ \mathfrak{r}_i\in \mathsf{N}_i\}$. In the same fashion, we have a natural bijection $\gamma: \{\mathfrak{m}_1\dots\mathfrak{m}_n\ |\ \mathfrak{m}_i\in \mathsf{M}_i\}\rightarrow W_n$, induced by the map $B_k\mapsto \R_k$, $T_i\mapsto \s_i$.\smallbreak
Let $w\in W_n$, and $\eta:W_n\rightarrow S_n$ the natural projection defined by $\R_1\mapsto 1$ and $\s_i\mapsto s_i$. Since in $\E$ the action of $B_1$ over the elements $F_i$'s and $E_i$'s is trivial, that is, these commute, we have the next result
\begin{lemma}\label{action}Let $w\in W_n$, $I\in \p$ and $A\subseteq \mathbf{n}$, then
  \begin{itemize}
    \item[a)] $T_wE_IT_w^{-1}=E_{\bar{w}(I)}$
    \item[b)] $T_wF_AT_w^{-1}=F_{\bar{w}(A)}$
  \end{itemize}
where $\overline{w}:=\eta(w)$
\end{lemma}
\begin{proof}
  For a) see \cite[Corollary 1]{rh}, and for b) the result follows by applying the defining relations (\ref{perm}) and (\ref{actb}).
\end{proof}
\begin{corollary}\label{actionCn}
 Let $v\in \{\mathfrak{m}_1\dots\mathfrak{m}_n\ |\ \mathfrak{m}_i\in \mathsf{M}_i\}$, $I\in \p$ and $A\subseteq \mathbf{n}$, then
  \begin{itemize}
    \item[a)] $vE_I v_{-1}=E_w(I)$
    \item[b)] $vF_A v_{-1}=F_w(A)$
\end{itemize}
where $w=\eta\circ \gamma(v)$.
\end{corollary}

%\begin{remark}
%  \rm Let $w\in W_n$, by \ref{typeB} we can conclude that
%$T_w=\mathfrak{r}_1\dots \mathfrak{r}_n$ with $\mathfrak{r}_i\in \mathsf{N}_i$. Then, there is a natural bijection between the sets \{T_w\ |\ w\in W_n\} and \{\mathfrak{m}_1\dots \mathfrak{m}_n\}, which is given by changing
%\end{remark}

\subsection{} \textbf{Diagrams for $\E$.} We associate to each word in the algebra $\E$, a tied braid of type $\mathtt{B}_n$ according the following identifications. For $n\geq 1$ we associate the unit of $\E$ with the trivial braid of type $\mathtt{B}_n$, $B_1$ with the \lq\lq loop\rq\rq generator, and $F_j$ with the braid of type $\mathtt{B}_n$ whose has a tied between the fixed strand and the j--th moving strand. For $n\geq 2$, as in the classical case we associate $T_i$ with the usual braid generator, and $E_i$ with the $\mathtt{B}$--type braid whose has a tied between the i-th and i+1-st moving strand.  We can see this identification in the following figure

\begin{figure}[h]
  \centering
  \includegraphics{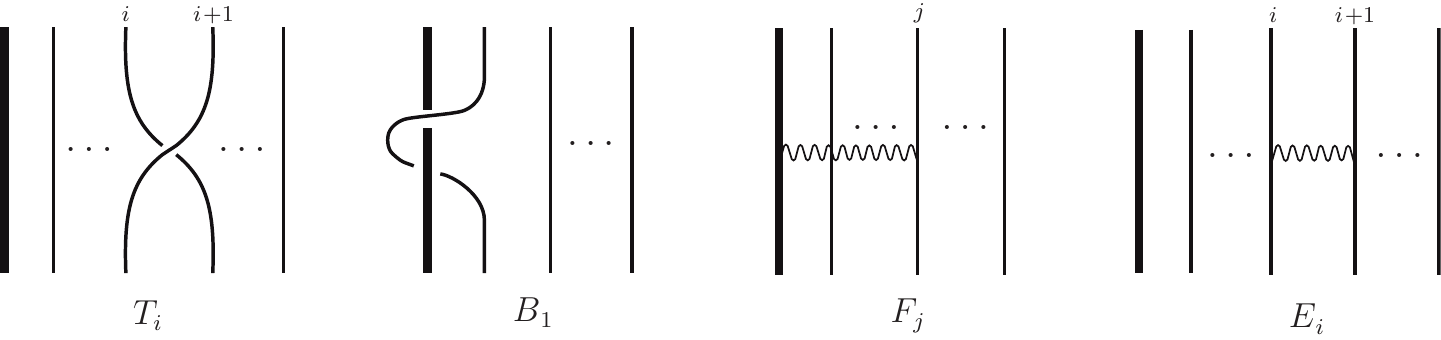}
  \caption{Diagrammatical interpretation of the generators of $\E$}\label{generatorsfig}
\end{figure}

We consider the multiplication of diagrams by concatenation, that is, given two diagrams $d_1$ and $d_2$, the multiplication $d_1d_2$ is the diagram that result from putting the diagram $d_2$ below to the diagram $d_1$, obtaining pictures like in Figure~\ref{relation}. Thus, we have a diagrammatical interpretation for every word in $\E$.

\begin{figure}[h]
  \centering
  \includegraphics{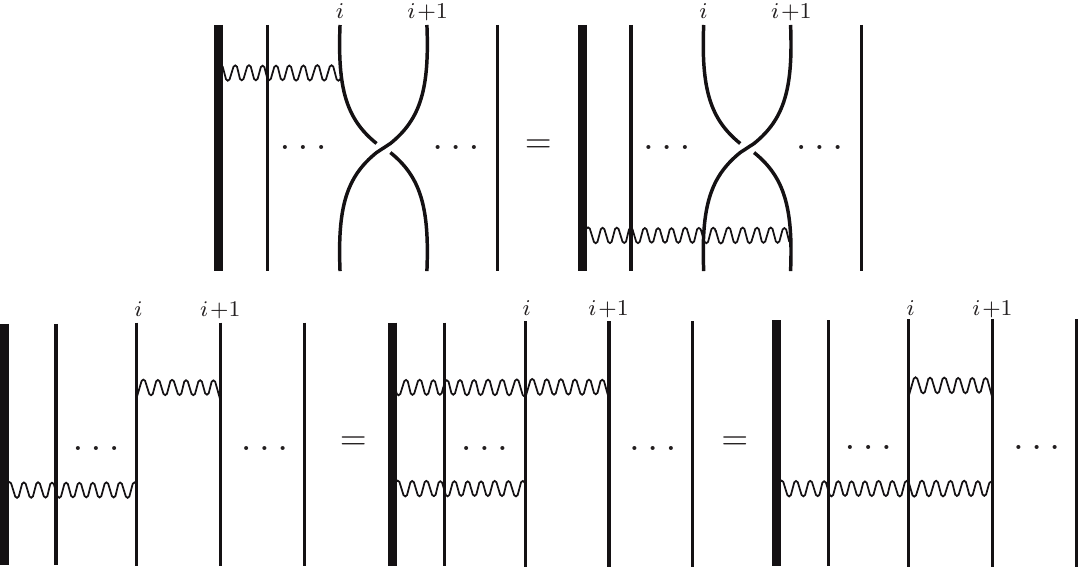}
  \caption{Relations (\ref{perm}) and (\ref{ef}) in terms of diagrams}\label{relation}
\end{figure}

\begin{remark}\rm
  The previous identification provide an epimorphism $\delta$, from the algebra $\E$ to an algebra of diagrams, which will be denoted by $\widetilde{\E}$. This algebra is generated by the elements in Figure~\ref{generatorsfig}, and satisfies the defining relations of $\E$ rewritten as diagram relations, for instance see Figure~\ref{relation}. As we would expected, the ties of the elements of $\widetilde{\E}$ keep all the properties explained in \cite[Section 6]{aijuMMJ}, that is, the \emph{elasticity, transparency} and \emph{transitivity}. For example, the elasticity and transparency properties, for ties involving the moving strands are inherited by the relations in common with the bt-algebra, and for the ties attached to the fixed strand are guaranteed by relation (\ref{perm}), see Figure~\ref{elast}. On the other hand, the transitivity property is a consequence of Proposition~\ref{bijection} proven in the next section.\\
In the same fashion, we can conjecture that the homomorphism $\delta$ is, in fact, an isomorphism. However, to give a formal proof of this fact, is in general a problem itself, for example see \cite{morton}. Therefore we will continue without proving it, since we would deviate of our main goal.
\end{remark}

\begin{figure}[h]
  \centering
  \includegraphics{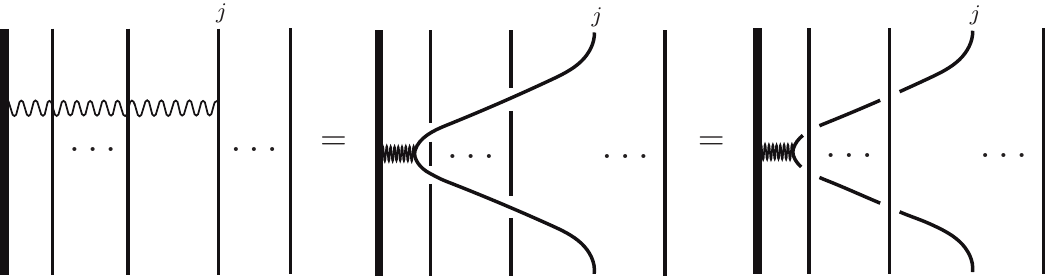}
  \caption{Equation (\ref{fi}) expressed in diagrams}\label{elast}
\end{figure}

\section{A linear basis for $\E$}
In this section we introduce two linear bases for the algebra $\E$. One of them will be used to define a Markov trace in the next section, and the other for proving the linear independence of the first one. Additionally we give a faithful tensorial representation of the algebra $\E$ based in the constructed for $\Y$ in \cite[Section 3]{fjl}.\smallbreak
We begin the section proving a technical result, which will be useful to set properly our base.\\

\noindent Let be $I,I'\in \p$ and $A,A'\in \n$. Note that, the elements $E_IF_A, E_{I'}F_{A'}\in \E$ can be equal even when $I\not=I'$ and $A\not=A'$. For instance, let $I=(\{2,3,5\},\{4,6\}), I'=(\{4,6\})\in \p$ and $A=\{3\}, A'=\{2,3,5\}$. Then, we have
\begin{eqnarray*}
  E_IF_A &=& E_{2,3}E_{2,5}E_{4,6}F_3 \\
   &=&F_3 E_{2,3}E_{2,5}E_{4,6}  \\
   &=&F_3F_2E_{2,5}E_{4,6} \\
   &=&F_2F_3F_5E_{4,6}  \\
   &=&E_{I'}F_{A'}
\end{eqnarray*}
by (\ref{efgen}). Using the same idea we can prove the following result
\begin{proposition}\label{bijection}
  The set $\mathcal{A}_n=\{E_IF_A\ |\ I\in \p, A\subseteq \mathbf{n}\} $ is parameterized by $\pcero$. That is, there is a bijection between the sets $\pcero$ and $\mathcal{A}_n$.
\end{proposition}
\begin{proof}

Let be $I=(I_1,\dots,I_m)$ a partition in $\pcero$. We define $\Psi:\pcero\rightarrow \mathcal{A}_n$
as follows
$$\Psi(I)=\left\{\begin{array}{ll}
           E_I  & \text{if $0\not\in I_j$ for all $1\leq j \leq m$}  \\
           E_{I\backslash I_k} F_{I_k^*} & \text{if $0\in I_k$ for some $1\leq k\leq m$ }
          \end{array}\right.$$

where $I\backslash I_r$ denote the partition obtained by removing the block $I_k$ from $I$. For the other hand, let be $I\in \p$, $A\in \mathbf{n}$ we define $\varphi:\mathcal{A}_n\rightarrow \pcero$ as follows
$$\varphi(E_IF_A)=I*A^0$$
where $A^0=A\cup \{0\}$ and $I$ is considered as a element of $\pcero$ (since $\mathcal{P}(\n)\subseteq\mathcal{P}(\n_0)$). It is not difficult to prove that $\varphi\circ \Psi=Id_{\pcero}$ and $\Psi\circ\varphi=Id_{\mathcal{A}_n}$. In fact, let say that $I=(I_1,\dots,I_k)$, without lose generality consider that $I_{1},\dots, I_{r}$ are the blocks of $I$ that have intersection with $A$, and $I_{r+1},\dots, I_{k}$ the blocks that don't have. By definition, we have that
$$I*A^0=(C,I_{r+1},\dots, I_{k})$$
where $C=A^0\cup I_{1}\cup \dots\cup I_{r} $. Thus, $\Psi(I*A^0)=F_{C^*}E_{I\backslash A}$ where $I\backslash A$ is the partition $(I_{r+1},\dots, I_{k})$ in $\p$.\\

Then, it is enough to prove that
\begin{equation}\label{*}
F_{C^*}E_{I\backslash A}=E_IF_A.
\end{equation}

First, recall that for any block $I_r$, we have that
$$E_{I_r}=\prod_{j\not=i_0, j\in I_r} E_{i_0,j},\qquad \text{where $i_0={\rm min}I_r$}$$
Thus, (\ref{efgen}) implies that
$$E_{I_r}F_j=E_{I_r}F_{i_0}$$
for any $j\in A\cap I_r$. Therefore, more generally using (\ref{idempotentf}) we have
\begin{eqnarray*}
  E_{I_r}F_A &=& E_{I_r}F_{i_0}F_{A\backslash I_r} \\
   &=& \prod_{j\not=i_0} E_{i_0,j}F_{i_0}F_{A\backslash I_r} \\
   &=& \prod_{j\not=i_0}F_{i_0}F_j F_{A\backslash I_r} \\
  &=& F_{I_r} F_{A\backslash I_r}=F_{A\cup I_r}
\end{eqnarray*}
Finally, (\ref{*}) follows by using this argument for every block that has intersection with $A$. The converse can be proven from an analogous way.
\end{proof}

From now on, given a partition $I\in \pcero$ we will denote the element $\Psi(I)$  by $EF_{I}$.

\begin{proposition}\label{span}
 The set $\mathcal{B}_n=\{EF_IT_w\ |\ I\in \pcero \text{ {\rm and} } w\in W_n\} $ spans the algebra $\E$.
\end{proposition}
\begin{proof}
We proceed as in \cite[Proposition 4]{fjl}, that is, we will prove by induction that $\mathfrak{B}_n$, the  linear subspace of $\E$ spanned by $\mathcal{B}_n$ is equal to $\E$. The assertion is clear for $n=1$. Assume now that $\mathcal{E}_{n-1}^{\mathtt B}$ is spanned by $\mathfrak{B}_{n-1}$. Notice that $1\in  \mathfrak{B}_n$. This fact and proving that $\mathfrak{B}_n$ is a right ideal, implies the proposition. Now, we deduce that $\mathfrak{B}_n$ is a right ideal from the hypothesis induction and  Lemma \ref{relaciones'}. Indeed, let $EF_IT_w$ a element in $\mathcal{B}_n$ and $Y$ a generator of $\E$, and suppose that $n\in {\rm Supp}(I)$, we have that $EF_IT_wY$ is equal to

$$E_{n,i}EF_{I\backslash n}x\mathbb{T}_{n,k}^{\pm}Y \quad\text{or}\quad F_nEF_{I\backslash n}x\mathbb{T}_{n,k}^{\pm}Y,\ \text{where $x\in \langle B_1,T_1\dots,T_{n-2}\rangle$, and $i={\rm min}(I_r)$, with $n\in I_r$.}$$
depending if $n$ belongs to $I_0$ or not. Then, using the relations from Lemma \ref{relaciones'} we can convert this expressions in a linear combination of elements of the form
$$E_{n,i}(X)\mathbb{T}_{n,k'}^{\pm} \quad\text{or}\quad F_n(X)\mathbb{T}_{n,k'}^{\pm},\qquad \text{ where $X\in \mathcal{E}_{n-1}^{\mathtt B}$.}$$
thus, the result follows by applying the induction hypothesis. The case when $n\not \in {\rm Supp}(I)$ will be omitted, since it follows by analogous way.

%Let be $X$ a word in the generators of $\E$, without loose generality, we can assume that X doesn't contains powers of the generators. Now, using the Lemma~\ref{action} we can move all the elements $E_i$'s and $F_j$'s to the left side, then we obtain a word of the form
% $$E_IF_AX'$$
%  for some $I\in \p$, $A\in \mathbf{n}$, and with $X'$ a word in the generators $T_i$'s and $B_1$. Since we know by \cite{fjl} that we can express $X'$ as linear combination of elements $T_w$ {\color{red}(eventually can appear elements $F_j$'s or $E_i$'s by applying the quadratic relations)} with $w\in W_n$ (see \cite[Section 4.1]{fjl}), the result follows.
\end{proof}
% Then, we can rewrite the set $\mathcal{B}_n$ as follows
%$$\mathcal{B}_n=\{EF_{I}T_w\ |\ I\in \pcero, A\in \mathbf{n}, \text{ {\rm and} } w\in W_n\} $$
%Moreover, we have that $|\mathcal{B}_n|=B_{n+1}2^nn!$.\\

Now, we will focus in proving the linear independence of the set $\mathcal{B}_n$, to achieve that, we will readjust the arguments used by Ryom-Hansen in \cite{rh}. For that, we use a tensorial representation of $\E$, which is obtained by restricting the representation of $\Y$ constructed in \cite{fjl}, considering $d=n+1$, to the subalgebra $\varphi_n(\E)$. \smallbreak

More precisely, let $V$ be a  ${\Bbb K}$--vector space with basis $\mathfrak{B}=\{v_i^r\, ;\,  i\in X_n,\, 0\leq r\leq n\} $.  As usual we denote by $\mathfrak{B}^{\otimes k}$ the natural  basis
of  $V^{\otimes k}$ associated to $\mathfrak{B}$. That is, the elements of $\mathfrak{B}^{\otimes k}$  are of the form:
$$
v_{i_1}^{m_1}\otimes\cdots\otimes v_{i_k}^{m_k}
$$
where $(i_1, \ldots , i_k)\in X_n^k$ and $(m_1, \ldots , m_k)\in \mathbf{n}^k$.\smallbreak

We define the endomorphisms $\mathbf{F}, \mathbf{B}:V\rightarrow V$ by:
$$
(v_{i}^{r})\mathbf{B}=
\left\{\begin{array}{cl}
 v_{-i}^{r} & \text{for}\quad i>0\ \text{and}\ r=0,\\
 v_{-i}^{r}+(\V-\V^{-1})v_{i}^{r}  &\text{for}\quad i<0\ \text{and}\ r=0,\\
  v_{-i}^{r} & \text{for}\quad r\not=0.
\end{array} \right.
$$
and
$$
(v_{i}^{r})\mathbf{F}=\left\{\begin{array}{cc}
 0 & r>0\\
 v_{i}^{r} & r=0\end{array}\right.$$

On the other hand we define $\mathbf{T}, \mathbf{E}:V\otimes V\rightarrow V\otimes V$ by
$$
(\vnor)\mathbf{T}=\left\{\begin{array}{cl}
\U\vinv & \text{ for}\quad i=j\  \text{and } \  r=s,\\
\vinv & \text{ for}\quad i< j\  \text{and } \  r=s,\\
\vinv+ (\U-\U^{-1})\vnor & \text{ for}\quad i>j\ \text{and }\  r=s,\\
\vinv & \text{ for}\quad  r\not=s.
\end{array} \right.
$$
 and
$$
(v_i^r\otimes v_j^s)\mathbf{E}=\left\{\begin{array}{cc}
 0 & r\not=s\\
 v_i^r\otimes v_j^s & r=s
\end{array} \right.
$$

For all $1\leq i \leq n-1$, $1\leq j\leq n$ we extend these endomorphisms to the endomorphisms $\mathbf{E}_i$, $\mathbf{T}_i$, $\mathbf{B}_1$, $\mathbf{F}_j$ of the $n$--th tensor power $V^{\otimes n}$ of $V$, as follows:
$$
\begin{array}{ll}
 \mathbf{E}_i:=1_V^{\otimes(i-1)}\otimes \mathbf{E}\otimes  1_V^{\otimes(n-i-1)}  &,\quad  \mathbf{B}_1:=\mathbf{B}\otimes 1_V^{\otimes(n-1)}, \\
  \mathbf{T}_i:=1_V^{\otimes(i-1)}\otimes \mathbf{T} \otimes 1_V^{\otimes(n-i-1)} &,\quad  \mathbf{F}_j:=1_V^{\otimes(j-1)}\otimes \mathbf{F}\otimes  1_V^{\otimes(n-j)}
\end{array}
$$

where $1_V^{\otimes k}$ denotes the endomorphism identity of $V^{\otimes k}$. \smallbreak

\begin{theorem}\label{tensor}
The mapping   $B_1\mapsto \mathbf{B}_1$, $T_i\mapsto \mathbf{T}_i$, $E_i\mapsto \mathbf{E}_i$ and $F_i\mapsto \mathbf{F}_i$ defines a representation $\Phi$ of $\E$  in ${\rm End}(V^{\otimes n})$.
\end{theorem}
\begin{proof}
  It is a consequence of \cite[Theorem 1]{fjl}.
\end{proof}

Further, we have
\begin{proposition}\label{aplicacion} {\rm (See \cite[Proposition 3]{fjl})}
Let $w\in W_n$ parameterized by $(m_1,\dots,m_n)\in X_n^n$. Then
$$
(v_1^{r_1}\otimes \dots \otimes v_n^{r_n})\Phi_{w}=v_{m_1}^{r_{\vert m_1\vert}}\otimes \dots \otimes v_{m_n}^{r_{\vert m_n\vert}}.
$$
where $\Phi_w$ denotes $\Phi(T_w)$.
\end{proposition}

Let $I=(I_1,\dots,I_m)\in \p$ and $A\subseteq \n$, we will denote $\Phi(E_I)$ and $\Phi(F_A)$ by $\mathbf{E}_I$ and $\mathbf{F}_A$ respectively. We know by \cite{fjl} that $\mathbf{E}_I$ acts over $V^{\otimes n}$ as follows
\begin{equation}\label{actE}
  (\vdimn)\mathbf{E}_I=\left\{\begin{array}{ll}
 0 & \text{if there exist $i,j,k$ such that $i,j\in I_k$ y $r_j\not=r_j$,}\\
 \vdimn & \text{otherwise}
\end{array} \right.
\end{equation}

In the same fashion, it is not difficult to prove that
\begin{equation}\label{actF}
 (\vdimn)\mathbf{F}_A=\left\{\begin{array}{ll}
 0 & \text{if there exist $i\in A$ such that $r_i\not=0$}\\
 \vdimn & \text{otherwise}
\end{array} \right.
\end{equation}
Now, we have all the necessary to prove the main result of this section
\begin{theorem}\label{basis}
  The set $\mathcal{B}_n$ is a basis of $\E$. In particular, the dimension of $\E$ is $b_{n+1}2^nn!$
\end{theorem}
\begin{proof}
  We only have to prove that $\mathcal{B}_n$ is a linear independent set, since it was already proven in Proposition~\ref{span} that it spans $\E$. Let be $I=(I_0,I_1,\dots,I_m)$ a element in $\pcero$ considering the single blocks in its expression. Without lose of generality set $I_0$ as the block that contains 0, then, we define $v^I\in V^{\otimes n}$ as follows
$$v^I:=v_1^{r_1}\otimes \dots \otimes v_n^{r_n},\quad \text{with $r_i=k$ if $i\in I_k$}$$
Suppose now that
\begin{equation}\label{suma}
  \sum_{J\in \pcero , w\in W_n} \lambda_{J,w}EF_{J}T_w=0
\end{equation}
Then, given $I\in \pcero$, if we apply $\Phi$ and evaluate $v^I$ in (\ref{suma}), we will obtain

$$ \sum_{J\in \pcero , w\in W_n} \lambda_{J,w}(v^I)(\mathbf{EF}_{I}\Phi_w)  = 0$$
thus, using (\ref{actE}) and (\ref{actF}) we have
\begin{eqnarray}
  \sum_{ w\in W_n} \lambda_{I,w}(v^I)\Phi_w  &=& 0 \nonumber\\
 \sum_{ w\in W_n} \lambda_{I,w}(v_1^{r_1}\otimes \dots \otimes v_n^{r_n})\Phi_w  &=& 0 \nonumber \\
 \sum_{} \lambda_{I,w}v_{m_1}^{r_{\vert m_1\vert}}\otimes \dots \otimes v_{m_n}^{r_{\vert m_n\vert}}  &=& 0\label{sumateo}
\end{eqnarray}
with $(m_1,\dots, m_n)$ running in $X_n^n$. But, this elements are L.I in $V^{\otimes n}$, then $\lambda_{I,w}=0$ for all $w\in W_n$. Finally as $I$ was picked arbitrarily the result follows.
\end{proof}

\begin{corollary}
  The representation $\Phi$ is faithful.
\end{corollary}

\begin{remark}\label{tower}
  Since  $\mathcal{B}_{n}\subseteq \mathcal{B}_{n+1}$, it follows that  $\E\subseteq \mathcal{E}_{n+1}^{\mathtt{B}}$, for all $n\geq 1$. Thus, by taking $\mathcal{E}_{0}^{\mathtt{B}}:= {\Bbb K}$, we have the following tower of algebras.
\begin{equation}\label{towereq}
  \mathcal{E}_0^{\B} \subseteq \mathcal{E}_1^{\B}\subseteq \cdots \subseteq \E \subseteq \mathcal{E}_{n+1}^{\mathtt{B}} \subseteq \cdots
\end{equation}
\end{remark}

\begin{remark}\rm
  Recall that in the original definition of $\Phi$ in \cite{fjl}, $V$ is considered as the vector space with basis $\mathcal{B}=\{v_i^r\, ;\,  i\in X_n,\, 0\leq r\leq d-1\} $. Thus, the condition $d=n+1$ it is essential for proving Theorem~\ref{basis}, in fact, if we suppose $d\leq n$, we could obtain a sum of linear dependent elements in Eq. (\ref{sumateo}). Moreover, to get a sum of linear independent elements, it is enough to consider $d\geq n+1$, which is contained in the next result.
\end{remark}

\begin{corollary}
  Suppose that $d\geq n+1$. Then the homomorphism $\varphi_n:\E \mapsto \Y$ defined in Proposition~\ref{homoEY} is an embedding.
\end{corollary}
\begin{proof}
  It is enough to prove that the set $A=\varphi_n(\mathcal{B}_n)$ is linear independent in $\Y$. Now, we know that $\Phi$ is faithful (\cite[Corollary ]{fjl}), and we proved in Theorem~\ref{basis} that $\Phi(A)$ is L.I, then, the result follows.
\end{proof}

\begin{corollary}
  The set $\mathcal{C}_n=\{EF_I( {\mathtt m}_1\dots {\mathtt m}_n) \ |\ I\in \pcero, {\mathtt m}_i\in \mathsf{M}_i\}$ is also a basis for $\E$.
\end{corollary}
\begin{proof}
  We can prove that $\mathcal{C}_n$ spans $\E$ analogously to Proposition~\ref{span}, but this time using the relations given in Lemma~\ref{relaciones}. The linear independence is guaranteed by cardinality.
\end{proof}

\section{Markov Trace in $\E$}
In this section we prove that $\E$ supports a Markov Trace. For that, we use the method of relative traces, cf.\cite{aijuMMJ,chpoIMRN}, which consists in construct a family of linear maps $\vartheta_n:\E\rightarrow \mathcal{E}_{n-1}^{\B}$, which gives step by step the desired Markov properties. Specifically, these properties are guaranteed by three key results, in our case these are Lemmas~\ref{lemmatraza1}, Lemma~\ref{commT} and (ii) Lemma~\ref{commE} (ii), which are essential to prove that the trace defined by $\mathtt{tr}_n=\vartheta_1\circ \dots \circ \vartheta_n$ is a Markov trace (Theorem~\ref{Markovtrace}). %Thus, the remaining lemmas are fixed by following the ideas used in \cite{aijuMMJ}, \cite{fjl} and \cite{chpoIMRN}. and these are only devoted to prove the key lemmas, which, at the same time, \\

\subsection{}From now on, we fix the parameters $\x,\W, \mathsf{y},\mathsf{z}\in \mathbb{K}$ and we consider $w\in \mathcal{C}_n$ expressed on the form $w=\w EF_I$ (which is possible by Corollary~\ref{actionCn}). Let $\mathbb{L}:=\mathbb{K}(\x,\W,\y,\z)$, then, when it is needed, we will consider $\x,\W, \mathsf{y},\mathsf{z}$ as variables, and work with the algebra $\E\otimes\mathbb{L}$, which, for simplicity, will be denoted by $\E$ again.

\begin{definition}\label{tracedef}
We set $\vartheta_1(B_1)=\y$, $\vartheta_1(F_1)=\x$ and $\vartheta_1(B_1F_1)=\W$. For $n\geq 2$, we define the linear map from $\E$ to $\mathcal{E}_{n-1}$ on the basis $\mathcal{C}_n$ as follows:

\begin{equation}\label{trace}
 \vartheta_n({\mathtt m}_1\cdots {\mathtt m}_{n-1}{\mathtt m}_n EF_I)=\left\{\begin{array}{ll}
    {\mathtt m}_1\cdots {\mathtt m}_{n-1}EF_I  & \text{\rm for ${\mathtt m}_n=1$, $n\not\in {\rm Supp}(I)$,}  \\
  \x {\mathtt m}_1\cdots {\mathtt m}_{n-1}EF_{I\backslash n}   & \text{\rm for ${\mathtt m}_n=1$, $n\in {\rm Supp}(I)$,}  \\
    \mathsf{z} {\mathtt m}_1\cdots {\mathtt m}_{n-1}\mathbb{{T}}_{n-1,k}^{\pm}EF_{\tau_{n,k}(I)}  & \text{\rm for ${\mathtt m}_n=\mathbb{T}_{n,k}^{\pm}$; $k<n$,} \\
   \mathsf{y} {\mathtt m}_1\cdots {\mathtt m}_{n-1}EF_I  &  \text{\rm for ${\mathtt m}_n=B_n$, $n\not\in {\rm Supp}(I)$,} \\
  \W {\mathtt m}_1\cdots {\mathtt m}_{n-1}EF_{I\backslash n}    & \text{\rm for ${\mathtt m}_n=B_n$, $n\in {\rm Supp}(I)$.}
     \end{array}\right.
\end{equation}

\vspace{3mm}
where $\tau_{n,k}(I)$ denotes the partition $(I*\{n,k\})\backslash n$.
\end{definition}
We begin proving some partitions properties, which will use frequently in the sequel

\begin{lemma}\label{partitionsprop}Let $\sigma \in S_n$ and $I\in \mathcal{P}(\n_0)$, then we have that
\begin{itemize}
  \item[(i)] $\sigma(I\backslash \{k\})=\sigma(I)\backslash \{\sigma(k)\}$, for some $k\in {\rm Supp}(I)$.
  \item[(ii)] $\sigma(I*\{j,k\})=\sigma(I)*\{\sigma(j),\sigma(k)\}$, for some $k,j\in \n$.
  \item[(iii)] $\sigma(\tau_{n,k}(I))=\tau_{\sigma(n),\sigma(k)}(\sigma(I))$ for some $k < n$.
\end{itemize}
\end{lemma}
  \begin{proof}
Let $I=(I_1,\dots,I_m) \in \mathcal{P}(\n_0)$, and suppose without lose of generality that $I_1$ contains $k$. Then we have that
$I\backslash\{k\}=(I_1',I_2,\dots, I_m)$
where $I_1'=I_1\backslash\{k\}$. Therefore
\begin{eqnarray*}
  \sigma(I\backslash \{k\}) &=& (\sigma(I_1'),\sigma (I_2),\dots, \sigma(I_m)) \\
   &=& (\sigma(I_1)\backslash \{\sigma(k)\},\sigma (I_2),\dots, \sigma(I_m))=\sigma(I)\backslash \{\sigma(k)\}
\end{eqnarray*}
and we have (i). For (ii) we only prove the case when $j,k\in {\rm Supp}(I)$, and these are in different blocks. Let $I_1$ y $I_2$ the blocks of $I$ that contains $j$ and $k$ respectively. Then, we have
$$I*\{j,k\}=(I_1\cup I_2, I_3, \dots, I_m)$$
therefore
\begin{eqnarray*}
  \sigma(I*\{n,k\}) &=&  (\sigma(I_1)\cup \sigma(I_2), \sigma(I_3), \dots, \sigma(I_m))\\
   &=& (\sigma(I_1), \sigma(I_2), \sigma(I_3), \dots, \sigma(I_m))*\{\sigma(n),\sigma(k)\}=\sigma(I)*\{\sigma(j),\sigma(k)\}
\end{eqnarray*}
Finally, (iii) is a consequence of (i) and (ii).
  \end{proof}

We would like to prove that (\ref{trace}) holds by considering $v\in \mathcal{E}_{n-1}^{\mathtt B}$ instead ${\mathtt m}_1\cdots {\mathtt m}_{n-1}$ in the formula. Having this in mind, we introduce some notation and we prove a technical lemma.
\smallbreak

Let $j>k\in \n$ we define the element $\sigma_{j,k}\in S_n$ by
$$\sigma_{j,k}=s_{j-1}\cdots s_k$$
where the $s_i$ denote the transposition $(i,i+1)$. Note that
\begin{equation}\label{actionsigmas}
  \sigma_{j,k}(i)=\left\{\begin{array}{cc}
   j  & \text{if $i=k$} \\
   i-1  &  \text{if $k<i\leq j$}\\
    i &  \text{otherwise}
  \end{array}\right. \quad \text{and} \quad \sigma_{j,k}^{-1}(i)=\left\{\begin{array}{cc}
   k  & \text{if $i=j$} \\
   i+1  &  \text{if $k\leq i< j$}\\
    i &  \text{otherwise}
  \end{array}\right.
\end{equation}

\begin{lemma}\label{equivalence}
  For $J=(J_1,\dots,J_r)\in \mathcal{P}(\n_0 \backslash\{n\})$ and $I=(I_1,\dots, I_s)\in \pcero$ the following equality holds.\begin{equation}\label{particion}
                (\sigma_{n,k}^{-1}(J)*I)*\{n,k\})\backslash n = \sigma_{n-1,k}^{-1}(J)*((I*\{n,k\})\backslash n)
               \end{equation}
\end{lemma}

\begin{proof}
  We only prove the case when $n,k \in {\rm Supp}(I)$, and these are in different blocks of $I$, since the other cases can be verified analogously. Let $J=(J_1,\dots,J_r)\in \mathcal{P}(\n_0 \backslash\{n\})$ and  $I=(I_1,\dots, I_s)\in \pcero$, without lose generality, we can suppose that $I_1$ and $I_2$ are the blocks of $I$ that contain $k$ and $n$ respectively. We proceed, distinguish cases.\smallbreak
\noindent\textsc{Case:} $n-1\not\in Supp(J)$\smallbreak
First, note that in this case the partitions $\sigma_{n,k}^{-1}(J)$ and $\sigma_{n-1,k}^{-1}(J)$ are the same, and it will be denoted by $A=(A_1,\dots, A_r)$. Moreover, we have that $n,k\not \in {\rm Supp(A)}$ by (\ref{actionsigmas}) and the fact that $n-1\not\in {\rm Supp}(J) $. Then, in these case (\ref{particion}) holds directly, since the operations $*\{n,k\}$ and $\backslash n$ just have influence over $I$.\smallbreak
\noindent\textsc{Case:} $n-1\in Supp(J)$\smallbreak
This time $\sigma_{n,k}^{-1}(J)$ and $\sigma_{n-1,k}^{-1}(J)$ are different, we will denote these by $A=(A_1,\dots, A_r)$ and $B=(B_1,\dots, B_r)$ respectively. Note that, $n\in {\rm Supp}(A)$,  $k\in {\rm Supp} (B)$,  $k\not\in {\rm Supp}(A)$ and  $n\not\in {\rm Supp} (B)$. Moreover, if $A_1$ and $B_1$ are the blocks of $A$ and $B$ that contain $n$ and $k$ respectively, we have that $A_i=B_i$ for $2\leq i\leq r$, and $A_1\backslash \{n\}=B_1\backslash\{k\}$.\smallbreak
Let $I_3,\dots,I_{t}$, with $t\leq s$, the blocks of $I$ that have intersection with $A_1\backslash \{n\}=B_1\backslash\{k\}$, and $I_{t+1},\dots,I_s$ those that don't have. We define the partitions
$$A'=(A_2,\dots,A_r),\quad B'=(B_2,\dots, B_r),\quad \text{and}\quad I'=(I_{t+1},\dots,I_s)$$
Then, for one side we have
\begin{eqnarray*}
  B*((I*\{n,k\})\backslash \{n\}) &=& B*((I_1\cup I_2)\backslash n, I_3,\dots ,I_t,I' ) \\
   &=& B*(I_1\cup (I_2\backslash \{n\}), I_3,\dots ,I_t,I' ) \\
   &=& (B_1\cup I_1\cup (I_2\backslash \{n\})\cup I_3\cup \dots \cup I_s, B'*I')\\
 &=&((B_1\backslash \{k\})\cup I_1\cup (I_2\backslash \{n\})\cup I_3\cup \dots \cup I_s, B'*I')
\end{eqnarray*}
On the other hand we have
\begin{eqnarray*}
  ((A*I)*\{n,k\}) &=& ((A_1\cup I_2\cup I_3\cup \dots \cup I_t,I_1,A'*I')*\{n,k\}) \backslash n\\
   &=& ((A_1\cup I_1 \cup I_2\cup I_3\cup \dots \cup I_t)\backslash n,A'*I') \\
   &=& ((A_1\backslash\{ n\})\cup I_1 \cup (I_2\backslash \{n\})\cup I_3\cup \dots \cup I_t,A'*I')
\end{eqnarray*}
since $A'=B'$  and $A_1\backslash \{n\}=B_1\backslash\{k\}$ the result follows.
\end{proof}

\begin{lemma}\label{multizq} For every $v\in \mathcal{E}_{n-1}^{\mathtt B}$ we have

  $$\vartheta_n(v{\mathtt m}_n EF_I)=\left\{\begin{array}{ll}
    vEF_I  & \text{\rm for ${\mathtt m}_n=1$, $n\not\in {\rm Supp}(I)$,}  \\
  \x vEF_{I\backslash n}   & \text{\rm for ${\mathtt m}_n=1$, $n\in {\rm Supp}(I)$,} \\
    \mathsf{z} v\mathbb{{T}}_{n-1,k}^{\pm}EF_{\tau_{n,k}(I)}  & \text{\rm for ${\mathtt m}_n=\mathbb{T}_{n,k}^{\pm}$; $k<n$,} \\
   \mathsf{y} vEF_I  &  \text{\rm for ${\mathtt m}_n=B_n$, $n\not\in {\rm Supp}(I)$,} \\
  \W vEF_{I\backslash n}    &  \text{\rm for ${\mathtt m}_n=B_n$, $n\in {\rm Supp}(I) $.}
     \end{array}\right.
\vspace{3mm}$$
\end{lemma}
\begin{proof}
  By the linearity of the trace is enough prove the statement for $v\in \mathcal{C}_{n-1}$. The cases when ${\mathtt m}_n=1$ can be proven easily. \smallbreak

 For case ${\mathtt m}_n=\mathbb{T}_{n,k}^{\pm}$ with $k<n$, we have

\begin{eqnarray*}
  \vartheta_n(v\mathbb{T}_{n,k}^{\pm}EF_I) &=& \vartheta_n({\mathtt m}_1\cdots {\mathtt m}_{n-1}EF_J\mathbb{T}_{n,k}^{\pm}EF_I) \\
   &=&\vartheta_n({\mathtt m}_1\cdots {\mathtt m}_{n-1}\mathbb{T}_{n,k}^{\pm}EF_{\sigma_{n,k}^{-1}(J)} EF_I)   \\
   &=& \vartheta_n({\mathtt m}_1\cdots {\mathtt m}_{n-1}\mathbb{T}_{n,k}^{\pm}EF_{\sigma_{n,k}^{-1}(J)*I})\\
&=&\z\ {\mathtt m}_1\cdots {\mathtt m}_{n-1}\mathbb{T}_{n-1,k}^{\pm}EF_{\tau_{n,k}(\sigma_{n,k}^{-1}(J)*I)}
\end{eqnarray*}
On the other hand, we have
\begin{eqnarray*}
 \z\ v\mathbb{T}_{n-1,k}^{\pm}EF_{\tau_{n,k}(I)}  &=&  \\
   &=& \z\ {\mathtt m}_1\cdots {\mathtt m}_{n-1}EF_J\mathbb{T}_{n-1,k}^{\pm} EF_{\tau_{n,k}(I)} \\
   &=& \z\ {\mathtt m}_1\cdots {\mathtt m}_{n-1}\mathbb{T}_{n-1,k}^{\pm}EF_{\sigma_{n-1,k}^{-1}(J)} EF_{\tau_{n,k}(I)}\\
&=& \z\ {\mathtt m}_1\cdots {\mathtt m}_{n-1}\mathbb{T}_{n-1,k}^{\pm}EF_{\sigma_{n-1,k}^{-1}(J)*\tau_{n,k}(I)}
\end{eqnarray*}
%   after \\: \hline or \cline{col1-col2} \cline{col3-col4} ...
%   $=$&$$ &$=$& \\
%    =&$$&$$  &\\
%   =&$$ &=$$ & \\
%=&$$ &=$$ &

Then the result follows by Lemma~\ref{equivalence}.\\

Finally, we suppose that  ${\mathtt m}_n=B_n$. We only prove the case when $n\in {\rm Supp}(I) $,  since the opposite case can be verified by an analogous way. Then, we have
\begin{eqnarray*}
  \vartheta_n(vB_nEF_I) &=& \vartheta_n(\w EF_J B_nEF_I)\\
   &=&\vartheta_n(\w B_n EF_J EF_I ) \\
   &=&\vartheta_n(\w B_n EF_{I*J})   \\
   &=& \W \w EF_{(I*J)\backslash n} \\
&=& \W \w EF_{(I\backslash n)*J}\\
&=&  \W \w EF_J EF_{I\backslash n}=\ \W v EF_{I\backslash n}
\end{eqnarray*}

\end{proof}

The following lemmas contain several computations analogous to the proved in \cite[Section 5]{aijuMMJ} for the bt--algebra. Therefore, although we work with a different quadratic relation, we will omit some computations in the following proofs, since these can be obtained by passing through the automorphism induced by the change of generators given in Remark~\ref{changvar}.

\begin{lemma}\label{lemmatraza1}
  For all $X,Z\in \mathcal{E}_{n-1}^{\mathtt B}$ and $Y\in \E$, we have:
\begin{itemize}
\item[(i)] $\vartheta_n(YZ)=\vartheta_n(Y)Z$
\item[(ii)] $\vartheta_n(XY)=X\vartheta_n(Y)$
\end{itemize}
\end{lemma}
\begin{proof}
  For proving claim (i) notice that, due to the linearity  of $\vartheta_n$, we can suppose that $ Z$ is a  defining generator of $\mathcal{E}_{n-1}^{\mathtt{B}}$ and $Y={\mathtt m}_1{\mathtt m}_2\dots {\mathtt m}_n EF_I$, with ${\mathtt m}_i\in \mathsf{M}_{i}$ and $I\in \pcero$. To prove the claim  we shall distinguish the $Y$'s according to the possibilities of ${\mathtt m}_n$, and if $n$ belongs to ${\rm Supp}(I)$ or not.\\

First, note that for $Z=E_i, F_j$ with $1\leq i\leq n-2$ and $1\leq j\leq n-1$, claim (i) holds easily for any choice of $\m_n$. Also, when ${\mathtt m}_n=1$ and $n\not\in {\rm Supp}(I)$, since $\vartheta_n$ acts like the identity. Now, we proceed to study the remaining cases.\\

\noindent \textsc{Case: }${\mathtt m}_n=1$, $n\in{\rm Supp}(I)$\smallbreak
If we consider $Z=T_j, E_j$ the result follows by \cite[Lemma~2]{aijuMMJ}. And for $Z=B_1, F_1$ the results follows by Lemma~\ref{multizq} and the fact that $B_1$ commutes with $EF_I$. \\

\noindent \textsc{Case:} ${\mathtt m}_n=B_n$, $n\in{\rm Supp}(I)$\smallbreak
 First suppose that $Z=T_j$ for $j\in \{1,\dots,n-2\}$
\begin{eqnarray*}
  \vartheta_n({\mathtt m}_1\cdots {\mathtt m}_{n-1}B_nEF_{I}T_j) &=& \vartheta_n({\mathtt m}_1\cdots {\mathtt m}_{n-1} B_n T_j EF_{s_j(I)})  \\
   &=& \vartheta_n({\mathtt m}_1\cdots {\mathtt m}_{n-1} T_j B_n EF_{s_j(I)}) \\
   &=& \W {\mathtt m}_1\cdots {\mathtt m}_{n-1} T_j EF_{s_j(I)\backslash n)} \\
\end{eqnarray*}
On the other hand, $\vartheta_n({\mathtt m}_1\cdots {\mathtt m}_{n-1}B_nEF_{I})T_j=\W {\mathtt m}_1\cdots {\mathtt m}_{n-1} T_j EF_{s_j(I\backslash n)}$. Thus, since $s_j$ doesn't act over $n$, we have that $s_j(I)\backslash n=s_j(I\backslash n)$, and the result follows. \smallbreak

 For $Z=B_1$, we have
\begin{eqnarray*}
  \vartheta_n({\mathtt m}_1\cdots {\mathtt m}_{n-1}B_nEF_{I}B_1) &=&\vartheta_n({\mathtt m}_1\cdots {\mathtt m}_{n-1}B_n B_1 EF_{I})  \\
   &=&\vartheta_n({\mathtt m}_1\cdots {\mathtt m}_{n-1}B_1B_nEF_{I}+ (\U-\U^{-1}){\mathtt m}_1\cdots {\mathtt m}_{n-1}\lambda_n  EF_{I})\\
 &=& \W {\mathtt m}_1\cdots {\mathtt m}_{n-1}B_1EF_{I\backslash n}+ \vartheta_n( (\U-\U^{-1}){\mathtt m}_1\cdots {\mathtt m}_{n-1}\lambda_n EF_{I}) \\
&=& \W {\mathtt m}_1\cdots {\mathtt m}_{n-1}EF_{I\backslash n}B_1+ (\U-\U^{-1})\vartheta_n({\mathtt m}_1\cdots {\mathtt m}_{n-1}\lambda_n EF_{I})\\
\end{eqnarray*}
by using ii) Lemma~\ref{relaciones}, where $\lambda_n=[B_1T_{1}^{-1}\dots T_{n-2}^{-1} T_{n-1}\dots T_1B_1E_{1,k}-T_{1}^{-1}\dots T_{n-2}^{-1}T_{n-1}\dots T_1B_1^2E_{1,k}]$.
Then, it is enough to prove that $\mathtt{A}=\vartheta_n( {\mathtt m}_1\cdots {\mathtt m}_{n-1}\lambda_n EF_{I})=0$. In fact, we have
\begin{eqnarray*}
 \mathtt{A}  &=&  \vartheta_n( {\mathtt m}_1\cdots{\mathtt m}_{n-1}[B_1T_{1}^{-1}\dots T_{n-2}^{-1} \mathbb{T}_{n,1}^{-}E_{1,k}-T_{1}^{-1}\dots T_{n-2}^{-1}\mathbb{T}_{n,1}^{+}B_1^2E_{1,k}]EF_{I})\\
   &=& \vartheta_n( {\mathtt m}_1\cdots {\mathtt m}_{n-1}B_1T_{1}^{-1}\dots T_{n-2}^{-1} \mathbb{T}_{n,1}^{-}E_{1,k}EF_{I})-\vartheta_n({\mathtt m}_1\cdots {\mathtt m}_{n-1}T_{1}^{-1}\dots T_{n-2}^{-1}\mathbb{T}_{n,1}^{+}B_1^2E_{1,k}EF_{I}) \\
   &=& \vartheta_n( {\mathtt m}_1\cdots {\mathtt m}_{n-1}B_1T_{1}^{-1}\dots T_{n-2}^{-1} \mathbb{T}_{n,1}^{-}EF_{I*\{1,k\}})-\vartheta_n({\mathtt m}_1\cdots {\mathtt m}_{n-1}T_{1}^{-1}\dots T_{n-2}^{-1}\mathbb{T}_{n,1}^{+}EF_{I*\{1,k\}})-\\
&&(\V-\V^{-1}) \vartheta_n({\mathtt m}_1\cdots {\mathtt m}_{n-1}T_{1}^{-1}\dots T_{n-2}^{-1}\mathbb{T}_{n,1}^{-}EF_{I*\{0,1,k\}})\\
&=& \z {\mathtt m}_1\cdots {\mathtt m}_{n-1}B_1T_{1}^{-1}\dots T_{n-2}^{-1} \mathbb{T}_{n-1,1}^{-}EF_{\tau_{n,1}(I*\{1,k\})}- \z {\mathtt m}_1\cdots {\mathtt m}_{n-1}T_{1}^{-1}\dots T_{n-2}^{-1}\mathbb{T}_{n-1,1}^{+}EF_{\tau_{n,1}(I*\{1,k\})}-\\
&&\z (\V-\V^{-1}) {\mathtt m}_1\cdots {\mathtt m}_{n-1}T_{1}^{-1}\dots T_{n-2}^{-1}\mathbb{T}_{n-1,1}^{-}EF_{\tau_{n,1}(I*\{0,1,k\})}\\
&=& \z {\mathtt m}_1\cdots {\mathtt m}_{n-1}B_1^2EF_{\tau_{n,1}(I*\{1,k\})}- \z {\mathtt m}_1\cdots {\mathtt m}_{n-1}EF_{\tau_{n,1}(I*\{1,k\})}-\\
&&\z (\V-\V^{-1}) {\mathtt m}_1\cdots {\mathtt m}_{n-1}B_1EF_{\tau_{n,1}(I*\{0,1,k\})}
\end{eqnarray*}
Finally, expanding $B_1^2$, we obtain that $\vartheta_n({\mathtt m}_1\cdots {\mathtt m}_{n-1}\lambda_n EF_{I})=0$, since $\tau_{n,1}(I*\{1,k\})*\{0,1\}=\tau_{n,1}(I*\{0,1,k\})$. For the case  ${\mathtt m}_n=B_n$, $n\not\in{\rm Supp}(I)$, we can proceed analogously (we only have to put $\y$ instead $\W$, and omit the operation $\backslash n$ in the partition ).\\

\noindent\textsc{Case:} ${\mathtt m}_n=\mathbb{T}_{n,k}^{+}$, \smallbreak
For $Z=T_j$ with $j\in \{1,\dots,n-2\}$ the result follows analogously as in \cite[Lemma 1]{aijuMMJ} using the relations of Lemma~\ref{relaciones}. Suppose now, that $Z=B_1$, if $k>1$, then $B_1$ commutes with $\mathbb{T}_{n,k}^{+}$, therefore the result follows easily. If $k=1$, we have
\begin{eqnarray*}
  \vartheta_n(YZ)=\vartheta_n({\mathtt m}_1\cdots {\mathtt m}_{n-1}{T}_{n,k}^{+}EF_IB_1) &=&\vartheta_n({\mathtt m}_1\cdots {\mathtt m}_{n-1}{T}_{n,1}^{+}B_1 EF_I)   \\
   &=&\vartheta_n({\mathtt m}_1\cdots {\mathtt m}_{n-1}{T}_{n,1}^{-}EF_I)   \\
   &=& \z\ {\mathtt m}_1\cdots {\mathtt m}_{n-1}{T}_{n-1,1}^{-}EF_{\tau_{n,1}(I)} \\
 &=&\z\ {\mathtt m}_1\cdots {\mathtt m}_{n-1}{T}_{n-1,1}^{+}EF_{\tau_{n,1}(I)}B_1=\vartheta_n(Y)Z
\end{eqnarray*}

\noindent \textsc{Case:} ${\mathtt m}_n=\mathbb{T}_{n,k}^{-}$.\smallbreak
 For $Z=T_j$ with $j\in \{1,\dots, n-2\}$, the proof follows analogously as in \cite[Lemma~1]{aijuMMJ}, since the formula i) of Lemma~\ref{relaciones} coincides % in the sense of Remark -- (agregar)
with (22) of \cite{aijuMMJ}.\\

Finally, for $Z=B_1$ we have
\begin{eqnarray*}
  \vartheta_n(YZ) &=& \vartheta_n(\w\tmenosn EF_I B_1)  \\
   &=&\vartheta_n(\w\tmenosn B_1 EF_I)  \\
   &=& \vartheta_n(\w B_1\mathbb{T}^{-}_{n,k} EF_I ) + (\U-\U^{-1})[\vartheta_n (\w B_1T_{1}^{-1}\dots T_{k-2}^{-1} \mathbb{T}_{n,1}^- E_{1,k}EF_I) -\\
&& \vartheta_n( \w T_{1}^{-1}\dots T_{k-2}^{-1}\mathbb{T}_{n,1}^- B_1 E_{1,k}EF_{I}]\\
   &=&  \vartheta_n(\w B_1\mathbb{T}^{-}_{n,k} EF_I ) + (\U-\U^{-1})[\vartheta_n (\w B_1T_{1}^{-1}\dots T_{k-2}^{-1} \mathbb{T}_{n,1}^- EF_{I*\{1,k\}}) -\\
&& \vartheta_n( \w T_{1}^{-1}\dots T_{k-2}^{-1}\mathbb{T}_{n,1}^+ EF_{I*\{1,k\}}))-(\V-\V^{-1})\vartheta_n( \w T_{1}^{-1}\dots T_{k-2}^{-1}\mathbb{T}_{n,1}^- EF_{I*\{0,1,k\}})]\\
   &=&\z\ \w B_1\mathbb{T}^{-}_{n-1,k} EF_{\tau_{n,k}(I)}  + \z (\U-\U^{-1})[\w B_1T_{1}^{-1}\dots T_{k-2}^{-1} \mathbb{T}_{n-1,1}^- EF_{\tau_{n,1}(I*\{1,k\})} -\\
&&  \w T_{1}^{-1}\dots T_{k-2}^{-1}\mathbb{T}_{n-1,1}^+ EF_{\tau_{n,1}(I*\{1,k\})}-(\V-\V^{-1}) \w T_{1}^{-1}\dots T_{k-2}^{-1}\mathbb{T}_{n-1,1}^- EF_{\tau_{n,1}(I*\{0,1,k\})}]
\end{eqnarray*}

On the other hand,
\begin{eqnarray*}
  \vartheta_n(Y)Z &=& \z\ \w \tmenos EF_{\tau_{n,k}(I)} B_1  \\
   &=& \z\ \w \tmenos B_1 EF_{\tau_{n,k}(I)} \\
   &=& \z\ \w \tmenos B_1 EF_{\tau_{n,k}(I)} \\
   &=& \z\ \w B_1\mathbb{T}^{-}_{n-1,k} EF_{\tau_{n,k}(I)}  + \z (\U-\U^{-1})[\w B_1T_{1}^{-1}\dots T_{k-2}^{-1} \mathbb{T}_{n-1,1}^- EF_{\tau_{n,k}(I)*\{1,k\}} -\\
&&  \w T_{1}^{-1}\dots T_{k-2}^{-1}\mathbb{T}_{n-1,1}^+ EF_{\tau_{n,k}(I)*\{1,k\}}-(\V-\V^{-1}) \w T_{1}^{-1}\dots T_{k-2}^{-1}\mathbb{T}_{n-1,1}^- EF_{\tau_{n,k}(I)*\{0,1,k\}}] \\
   \end{eqnarray*}
clearly we have that $\tau_{n,k}(I)*\{1,k\}=\tau_{n,1}(I*\{1,k\})$ and $\tau_{n,1}(I*\{0,1,k\})=\tau_{n,k}(I)*\{0,1,k\}$, then, the result follows.

Finally (ii) is a direct consequence of Lemma~\ref{multizq}, since $X\w\in \mathcal{E}_{n-1}^{\mathtt B}$.
\end{proof}
\begin{corollary}   For all $X,Z\in \mathcal{E}_{n-1}^{\mathtt B}$ and $Y\in \E$, we have:
  $$\vartheta_n(XYZ)=X\vartheta_n(Y)Z.$$
\end{corollary}
\begin{proof}
  The proof is straightforward using the previous lemmas.
\end{proof}
\begin{lemma}\label{auxiliar}
  For $n\geq 2$, $X\in \mathcal{E}_{n-1}^\B$ and $Y\in \E$, we have
\begin{itemize}\label{commE}
  \item[(i)] $\vartheta_n(E_{n-1}X T_{n-1})=\vartheta_n(T_{n-1}X E_{n-1})$
  \item[(ii)] $\vartheta_{n-1}(\vartheta_{n}(E_{n-1}Y))=\vartheta_{n-1}(\vartheta_n(YE_{n-1}))$
\end{itemize}

\end{lemma}
\begin{proof}
As always, by linearity of the trace, we can consider $X$ and $Y$ in $\mathcal{C}_{n-1}$ and $\mathcal{C}_n$ respectively. Let $X=\wmenos {\mathtt m}_{n-1}EF_I$, with $I\in \mathcal{P}(\n_0\backslash \{n\})$, for proving (i) we will distinguish cases depending of the value of ${\mathtt m}_{n-1}$.\smallbreak
\noindent \textsc{Case:} ${\mathtt m}_{n-1}=1$. For one side, we have
\begin{eqnarray*}
  \vartheta_n(E_{n-1}X T_{n-1}) &=&  \vartheta_n(E_{n-1}\wmenos EF_I T_{n-1}) \\
   &=& \vartheta_n(\wmenos T_{n-1} EF_{s_{n-1}(I)}E_{n-1} ) \\
   &=& \vartheta_n(\wmenos T_{n-1} EF_{s_{n-1}(I)*\{n-1,n\}} \\
&=&\z \ \wmenos EF_{\tau_{n,n-1}(s_{n-1}(I)*\{n-1,n\})}
\end{eqnarray*}
On the other hand
\begin{eqnarray*}
 \vartheta_n(T_{n-1}X E_{n-1})  &=& \vartheta_n(T_{n-1}\wmenos EF_I E_{n-1}) \\
   &=&\vartheta_n(\wmenos T_{n-1} EF_{I*\{n-1,n\}} )  \\
   &=& \z\  \wmenos EF_{\tau_{n,n-1}(I*\{n-1,n\})}
\end{eqnarray*}
Now, if $n-1\not \in {\rm Supp}(I)$ the equality is clear, since $s_{n-1}(I)=I$. On the other hand, if $n-1 \in {\rm Supp}(I)$, it is not difficult to prove that $s_{n-1}(I)*\{n-1,n\}=I*\{n-1,n\}$.\\

\noindent\textsc{Case:} If ${\mathtt m}_{n-1}=\mathbb{T}_{n-1,k}^{\pm}$, we have
\begin{eqnarray*}
  \vartheta_n(E_{n-1}X T_{n-1})  &=& \vartheta_n(\wmenos \mathbb{T}_{n-1,k}^{\pm}E_{n,k} T_{n-1} EF_{s_{n-1}(I)}) \\
   &=& \vartheta_n(\wmenos \mathbb{T}_{n-1,k}^{\pm} T_{n-1} EF_{s_{n-1}(I)*\{n-1,k\}}) \\
   &=&\z\  \wmenos \mathbb{T}_{n-1,k}^{\pm} EF_{\tau_{n,n-1}(s_{n-1}(I)*\{n-1,k\})}
 \end{eqnarray*}
On the other hand
\begin{eqnarray*}
 \vartheta_n(T_{n-1}X E_{n-1}) &=& \vartheta_n(T_{n-1}\wmenos \mathbb{T}_{n-1,k}^{\pm} EF_{I}E_{n-1}) \\
   &=& \vartheta_n(\wmenos \mathbb{T}_{n,k}^{\pm} EF_{I*\{n-1,n\}}) \\
   &=&\z\  \wmenos\mathbb{T}_{n-1,k}EF_{\tau_{n,k}(I*\{n-1,n\})}
\end{eqnarray*}
and, it is easy to verify that the partitions from both cases are equal.\\

\noindent \textsc{Case:} If ${\mathtt m}_{n-1}=B_{n-1}$, we have
\begin{eqnarray*}
  \vartheta_n( E_{n-1}XT_{n-1})  &=&  \vartheta_n(E_{n-1}\wmenos B_{n-1}EF_IT_{n-1}) \\
   &=& \vartheta_n(\wmenos B_{n-1}T_{n-1}E_{n-1}EF_{s_{n-1}(I)}) \\
   &=& \vartheta_n(\wmenos B_{n-1}T_{n-1}EF_{s_{n-1}(I)*\{n-1,n\}}) \\
   &=&\z\ \wmenos B_{n-1}EF_{\tau_{n,n-1}(s_{n-1}(I)*\{n-1,n\})}
\end{eqnarray*}
On the other hand
\begin{eqnarray*}
 \vartheta_n(T_{n-1}X E_{n-1})  &=&  \vartheta_n(T_{n-1}\wmenos B_{n-1} EF_{I}E_{n-1})  \\
   &=& \vartheta_n(\wmenos T_{n-1}B_{n-1} EF_{I*\{n-1,n\}}) \\
   &=&\vartheta_n(\wmenos \mathbb{T}_{n,n-1}^{-} EF_{I*\{n-1,n\}})  \\
   &=& \z \ \wmenos B_{n-1}EF_{\tau_{n,n-1}(I*\{n-1,n\})})
\end{eqnarray*}
and we know by the first case that the partitions involved are equal, then we have already proved (i).\\

For proving (ii) we need more cases, since we have to apply two levels of the relative trace, then the result depends from the values of ${\mathtt m}_{n-1}$ and ${\mathtt m}_n$ of $Y=\wmenos {\mathtt m}_{n-1}{\mathtt m}_nEF_I\in \mathcal{C}_n$. \smallbreak
First note that for the cases $\m_n=1, \m_{n-1}=1$; $\m_n=1, \m_{n-1}=B_{n-1}$; $\m_n=B_n, \m_{n-1}=1$ and $\m_n=B_n, \m_{n-1}=B_{n-1}$, the result follows directly, since $E_{n-1}$ commute with $Y$ in each case, then we only have to analize five cases.\\

\noindent\textsc{Case:} If $\m_n=1$ and $\m_{n-1}=\mathbb{T}_{n-1,k}^{\pm}$.
\begin{eqnarray*}
  \vartheta_{n-1}(\vartheta_n(YE_{n-1})) &=& \vartheta_{n-1}(\vartheta_n(\wmenos\mathbb{T}_{n-1,k}^{\pm}EF_I E_{n-1})) \\
   &=& \vartheta_{n-1}(\vartheta_n(\wmenos\mathbb{T}_{n-1,k}^{\pm}EF_{I*\{n-1,n\}}) \\
   &=&\x\vartheta_{n-1}(\ \wmenos\mathbb{T}_{n-1,k}^{\pm}EF_{(I*\{n-1,n\})\backslash n)}) \\
   &=&\x\z\ \wmenos\mathbb{T}_{n-2,k}^{\pm}EF_{\tau_{n-1,k}(I_1)}
\end{eqnarray*}
where $I_1=(I*\{n-1,n\})\backslash n$. On the other hand
\begin{eqnarray*}
\vartheta_{n-1}(\vartheta_n(E_{n-1}Y))   &=& \vartheta_{n-1}(\vartheta_n(\wmenos E_{n-1}\mathbb{T}_{n-1,k}^{\pm} EF_I)) \\
   &=&\vartheta_{n-1}(\vartheta_n(\wmenos \mathbb{T}_{n-1,k}^{\pm} EF_{I*\{n,k\}}))  \\
   &=& \x\vartheta_{n-1}(\ \wmenos \mathbb{T}_{n-1,k}^{\pm} EF_{(I*\{n,k\})\backslash n})  \\
   &=&\x\z\ \wmenos \mathbb{T}_{n-2,k}^{\pm} EF_{\tau_{n-1,k}(I_2)}
\end{eqnarray*}
where $I_2= I*\{n,k\})\backslash n$. Further, we know by \cite[Lemma 2]{aijuMMJ} that $\tau_{n-1,k}(I_1)=\tau_{n-1,k}(I_2)$. Note that for case $\m_n=B_n$ and $\m_{n-1}=\mathbb{T}_{n-1,k}^{\pm}$ we have an analogous proof, since $E_{n,k}$ commutes with $B_{n}$, indeed, the only difference with this case is that when we apply the trace at level $n$ appear the parameter $\W$ instead $\x$.\\

\noindent\textsc{Case:} If $\m_n=\mathbb{T}_{n,k}^{\pm}$ and $\m_{n-1}=1$.
\begin{eqnarray*}
  \vartheta_{n-1}(\vartheta_n(YE_{n-1})) &=&  \vartheta_{n-1}(\vartheta_n(\wmenos\mathbb{T}_{n,k}^{\pm} EF_IE_{n-1}))\\
   &=&  \vartheta_{n-1}(\vartheta_n(\wmenos\mathbb{T}_{n,k}^{\pm}EF_{I*\{n-1,n\}}))\\
   &=& \z\ \wmenos\vartheta_{n-1}(\mathbb{T}_{n-1,k}^{\pm}EF_{I_1})
 %  &=&\z^2 \wmenos\mathbb{T}_{n-2,k}^{\pm}EF_{\tau_{n-1,k}(I_1))}
\end{eqnarray*}
where $I_1=\tau_{n,k}(I*\{n-1,n\})$. On the other hand
\begin{eqnarray*}
 \vartheta_{n-1}(\vartheta_n(E_{n-1}Y))  &=&  \vartheta_{n-1}(\vartheta_n(E_{n-1}\wmenos\mathbb{T}_{n,k}^{\pm} EF_I)) \\
   &=&\vartheta_{n-1}(\vartheta_n(\wmenos\mathbb{T}_{n,k}^{\pm} EF_{I*\{n,k\}}))  \\
   &=&\z\ \wmenos \vartheta_{n-1}(\mathbb{T}_{n-1,k}^{\pm} EF_{I_2})
   %&=& \z^2\ \wmenos\mathbb{T}_{n-2,k}^{\pm} EF_{\tau_{n-1,k}(I_2)}
\end{eqnarray*}
where $I_2=\tau_{n,k}(I*\{n,k\})$. First note that when $k=n-1$ the result follows directly, and if $k<n-1$ we obtain
$$ \vartheta_{n-1}(\mathbb{T}_{n-1,k}^{\pm} EF_{I_i})=\mathbb{T}_{n-2,k}^{\pm}EF_{\tau_{n-1,k}(I_i)}$$
Again by \cite[Lemma 2]{aijuMMJ} we have that $\tau_{n-1,k}(I_1)=\tau_{n-1,k}(I_2)$ and the result follows.\\

\noindent\textsc{Case:} If $\m_n=\mathbb{T}_{n,k}^{\pm}$ and $\m_{n-1}=B_{n-1}$.

\begin{eqnarray*}
 \vartheta_{n-1}(\vartheta_n(YE_{n-1}))  &=&   \vartheta_{n-1}(\vartheta_n(\wmenos B_{n-1} \mathbb{T}_{n,k}^{\pm} EF_I E_{n-1}))\\
 &=&\vartheta_{n-1}(\vartheta_n(\wmenos B_{n-1} \mathbb{T}_{n,k}^{\pm} EF_{I*\{n-1,n\}})) \\
   &=&\z\ \wmenos \vartheta_{n-1}( B_{n-1} \mathbb{T}_{n-1,k}^{\pm} EF_{I_1})
\end{eqnarray*}
On the other hand
\begin{eqnarray*}
 \vartheta_{n-1}(\vartheta_n(E_{n-1}Y))  &=&   \vartheta_{n-1}(\vartheta_n(E_{n-1}\wmenos B_{n-1} \mathbb{T}_{n,k}^{\pm} EF_I ))\\
&=& \vartheta_{n-1}(\vartheta_n(\wmenos B_{n-1} \mathbb{T}_{n,k}^{\pm} E_{n,k}EF_I ))\\
   &=&\z\ \wmenos \vartheta_{n-1}( B_{n-1} \mathbb{T}_{n-1,k}^{\pm} EF_{I_2})
\end{eqnarray*}
where $I_1=\tau_{n,k}(I*\{n-1,n\})$ and $I_2=\tau_{n,k}(I*\{n,k\}) $. Now, note that when $k=n-1$ the result is direct, then we can suppose that $k<n-1$, thus we have that

$$B_{n-1}\mathbb{T}_{n-1,k}^{\pm}=B_{n-1}T_{n-2}\mathbb{T}_{n-2,k}^{\pm}=\mathbb{T}_{n-1,n-2}^{-}\mathbb{T}_{n-2,k}^{\pm}=T_{n-2}B_{n-2}\mathbb{T}_{n-2,k}^{\pm}$$
Using this we obtain for one side
\begin{eqnarray*}
  \vartheta_{n-1}( B_{n-1} \mathbb{T}_{n-1,k}^{\pm} EF_{I_j}) &=&  \vartheta_{n-1}( \mathbb{T}_{n-1,n-2}^{-}\mathbb{T}_{n-2,k}^{\pm} EF_{I_j})\\
   &=& \vartheta_{n-1}( \mathbb{T}_{n-1,n-2}^{-} EF_{\sigma(I_j)})\mathbb{T}_{n-2,k}^{\pm} \\
 &=&\z\ B_{n-2} EF_{\tau_{n-1,n-2}(\sigma(I_j))}\mathbb{T}_{n-2,k}^{\pm}
\end{eqnarray*}
where $\sigma=\sigma_{n-2,k}$. Let see that the partitions are equal. let be $A_1,A_2,A_3$ the blocks of $I$ that contains $k,n-1$ and $n$ respectively, consider $I'$ as the partition that result by removing the blocks $A_1,A_2$ and $A_3$ from $I$. Then, we have
\begin{eqnarray*}
   I_1  &=&(I*\{n,n-1\})*\{n,k\})\backslash n= ((I*\{n,n-1, k\})\backslash n= I'*(A_1\cup A_2\cup A_3')  \\
    I_2 &=& ((I*\{n,k\})*\{n,k\})\backslash n= ((I*\{n,k\})\backslash n = I'*(A_1\cup A_3', A_2)
\end{eqnarray*}
where $A_3'=A_3\backslash \{n\}$. Therefore
\begin{eqnarray*}
   \sigma(I_1)  &=& \sigma(I')*(\sigma(A_1)\cup \sigma(A_2)\cup \sigma(A_3'))  \\
    \sigma (I_2) &=&\sigma(I')*(\sigma(A_1)\cup \sigma(A_3'), \sigma(A_2))
\end{eqnarray*}
now, note that $\sigma(k)=n-2$ and $ \sigma(n-1)=n-1$, then we have
\begin{eqnarray*}
  \tau_{n-1,n-2}(\sigma(I_1))  &=& (\sigma(I_2)*\{n-1,n-2\})\backslash n-1= \sigma(I')*(\sigma(A_1)\cup \sigma(A_2')\cup \sigma(A_3')) \\
   \tau_{n-1,n-2}(\sigma (I_2)) &=& (\sigma(I_1)*\{n-1,n-2\})\backslash n-1 = \sigma(I')*(\sigma(A_1)\cup \sigma(A_3')\cup \sigma(A_2'))
\end{eqnarray*}
where $A_2=A_2\backslash \{n-1\}$.\\

\noindent \textsc{Case:} If $\m_n=\mathbb{T}_{n,k}^{\pm}$ and $\m_{n-1}=\mathbb{T}_{n-1,j}^{\pm}$. Similarly as the last case, we have
\begin{eqnarray*}
   \vartheta_{n-1}(\vartheta_n(YE_{n-1}))  &=&   \vartheta_{n-1}(\vartheta_n(\wmenos \mathbb{T}_{n-1,j}^{\pm} \mathbb{T}_{n,k}^{\pm} EF_{I*\{n-1,n\}})) \\
   &=& \z\ \wmenos \vartheta_{n-1}(\mathbb{T}_{n-1,j}^{\pm} \mathbb{T}_{n-1,k}^{\pm} EF_{I_1})
\end{eqnarray*}
On the other hand
\begin{eqnarray*}
   \vartheta_{n-1}(\vartheta_n(E_{n-1}Y))  &=&   \vartheta_{n-1}(\vartheta_n(\wmenos E_{n-1}\mathbb{T}_{n-1,j}^{\pm} \mathbb{T}_{n,k}^{\pm} EF_{I})) \\
   &=&  \wmenos \vartheta_{n-1}(\vartheta_n(\mathbb{T}_{n-1,j}^{\pm} \mathbb{T}_{n-1,k}^{\pm} E_{a,b}EF_{I}))\\
&=& \z\ \wmenos \vartheta_{n-1}(\mathbb{T}_{n-1,j}^{\pm} \mathbb{T}_{n-1,k}^{\pm} EF_{I_2})
\end{eqnarray*}
where $I_1=\tau_{n,k}(I*\{n-1,n\})$, $I_2=\tau_{n,k}(I*\{a,b\})$ and
$$\{a,b\}=\left\{\begin{array}{ll}
            \{k,j\} & \text{if $j<k$}  \\
           \{k,j+1\}  & \text{if $j\geq k$}
          \end{array}\right.
$$
\noindent $\bullet$ \textsc{Subcase $k<n-1$:} First, if $j=n-2$ implies that $j\geq k$, therefore $\{a,b\}=\{j+1,k\}=\{n-1,k\}$, thus $I_1=I_2$, and the result follows. Then, we can suppose $j<n-2$, and we obtain
$$\mathbb{T}_{n-1,j}^{\pm} \mathbb{T}_{n-1,k}^{\pm}=T_{n-2}T_{n-3}\mathbb{T}_{n-3,j}^{\pm} T_{n-2}\mathbb{T}_{n-2,k}=T_{n-2}T_{n-3}T_{n-2}\mathbb{T}_{n-3,j}^{\pm}\mathbb{T}_{n-2,k}^{\pm}=T_{n-3}T_{n-2}\mathbb{T}_{n-2,j}^{\pm}\mathbb{T}_{n-2,k}^{\pm}$$
thus
\begin{eqnarray*}
\vartheta_{n-1}( \mathbb{T}_{n-1,j}^{\pm} \mathbb{T}_{n-1,k}^{\pm}EF_{I_1})   i&=&T_{n-3}\vartheta_{n-1}(T_{n-2}EF_{I_i'}) \mathbb{T}_{n-2,j}^{\pm}\mathbb{T}_{n-2,k}^{\pm} \\
   &=& \z \ T_{n-3}EF_{\tau_{n-1,n-2}(I_i')} \mathbb{T}_{n-2,j}^{\pm}\mathbb{T}_{n-2,k}^{\pm} \\
\end{eqnarray*}
where $I_i'=\sigma(I_i)$, with $\sigma=\sigma_{n-2,j}\sigma_{n-2,k}$, for $i=1,2$. And, it is known that $\tau_{n-1,n-2}(I_1')=\tau_{n-1,n-2}(I_2')$ by \cite[Lemma 2]{aijuMMJ}. \\

\noindent $\bullet$ \textsc{Subcase $k=n-1$:} We have that $\mathbb{T}_{n-1,k}^{\pm}=1$ or $B_{n-1}$, for the first case the result is direct. Then, suppose $\mathbb{T}_{n-1,k}^{\pm}=B_{n-1}$, we proceed with the positive case first, that is $\m_{n-1}=\mathbb{T}_{n-1,j}^+$. Note that
\begin{eqnarray*}
\mathbb{T}_{n-1,j}^+  B_{n-1} &=& T_{n-2}\cdots T_j T_{n-2}\cdots T_1 B_1 T_1^{-1}\cdots T_{n-2}^{-1} \\
   &=& T_{n-2}^2\cdots T_1 B_1 T_1^{-1}\cdots T_{n-2}^{-1}\mathbb{T}_{n-2,j}^+  \\
   &=&  T_{n-3}\cdots T_1 B_1 T_1^{-1}\cdots T_{n-2}^{-1}\mathbb{T}_{n-2,j}^+ + (\U-\U^{-1})B_{n-1}E_{n-2}\mathbb{T}_{n-2,j}^+\\
   &=& B_{n-2}T_{n-2}\mathbb{T}_{n-2,j}^+-(\U-\U^{-1})B_{n-2}E_{n-2}\mathbb{T}_{n-2,j}^++(\U-\U^{-1})B_{n-1}E_{n-2}\mathbb{T}_{n-2,j}^+
\end{eqnarray*}
Therefore we obtain
\begin{eqnarray*}
  \vartheta_{n-1}(\mathbb{T}_{n-1,j}^{+} B_{n-1} EF_{I_i}) &=& \vartheta_{n-1}(B_{n-2}T_{n-2}\mathbb{T}_{n-2,j}^+ EF_{I_i})- (\U-\U^{-1})\vartheta_{n-1}(B_{n-2}E_{n-2}\mathbb{T}_{n-2,j}^+EF_{I_i})+\\
 &&(\U-\U^{-1})\vartheta_{n-1}(B_{n-1}E_{n-2}\mathbb{T}_{n-2,j}^+EF_{I_i})\\
   &=& B_{n-2}\vartheta_{n-1}(T_{n-2} EF_{I_i'})\mathbb{T}_{n-2,j}^+- (\U-\U^{-1})B_{n-2}\vartheta_{n-1}(EF_{I_i'*\{n-1,n-2\}})\mathbb{T}_{n-2,j}^++\\
&&(\U-\U^{-1})\vartheta_{n-1}(B_{n-1}EF_{I_i'*\{n-1,n-2\}})\mathbb{T}_{n-2,j}^+\\
   &=&  \z B_{n-2} EF_{\tau_{n-1,n-2}(I_i')}\mathbb{T}_{n-2,j}^+- \x(\U-\U^{-1})B_{n-2}EF_{(I_i'*\{n-1,n-2\})\backslash n-1}\mathbb{T}_{n-2,j}^++\\
&& \W (\U-\U^{-1})B_{n-1}EF_{(I_i'*\{n-1,n-2\})\backslash n-1}))\mathbb{T}_{n-2,j}^+\\
\end{eqnarray*}
where $I_j'=\sigma(I_j)$ with $\sigma=\sigma_{n-2,j}$. First, note that by definition $(I_j'*\{n-1,n-2\})\backslash n-1=\tau_{n-1,n-2}(I_j')$, also we have $j<k$, since $k=n-1$. Therefore, we can deduce from the last case that $\tau_{n-1,n-2}(I_1')=\tau_{n-1,n-2}(I_2')$.\smallbreak
Finally, suppose that $\m_{n-1}=\mathbb{T}_{n-1,j}^-$. For (i) \cite[Lemma 8]{fjl} (taking $m=0$) we have that
$$\mathbb{T}_{n-1,j}^-B_{n-1}=B_{n-2}\mathbb{T}_{n-1,j}^-$$
Therefore
\begin{eqnarray*}
 \vartheta_{n-1}(\mathbb{T}_{n-1,j}^-B_{n-1})  &=& \vartheta_{n-1}(B_{n-2} \mathbb{T}_{n-1,j}^-EF_{I_i})  \\
   &=& B_{n-2}\vartheta_{n-1}( \mathbb{T}_{n-1,j}^-EF_{I_i}) \\
   &=& \z B_{n-2} \mathbb{T}_{n-2,j}^-EF_{\tau_{n-1,j}(I_i)}
\end{eqnarray*}
for $i=1,2$. Then, we have to verify that $\tau_{n-1,j}(I_1)=\tau_{n-1,j}(I_2)$. Since we are supposing $k=n-1$, we have that
\begin{eqnarray*}
  I_1 &=&I*\{n-1,n\}\backslash n  \\
   I_2 &=& I*\{n,j,n-1\} \backslash n
\end{eqnarray*}
and therefore $(I_1*\{n-1,j\})\backslash n-1=(I_2*\{n-1,j\})\backslash n-1$.
\end{proof}

\begin{lemma}\label{conjugation}
  For $n\geq 2$ and $X\in \mathcal{E}_{n-1}^{\mathtt B}$. We have
\begin{itemize}
  \item[(i)] $\vartheta_n(T_{n-1}XT_{n-1}^{-1})=\vartheta_{n-1}(X)=\vartheta_n(T_{n-1}^{-1}XT_{n-1})$
 \end{itemize}
\end{lemma}

\begin{proof}
Consider $X=\wmenos\m_{n-1}EF_I$, with $I\in \mathcal{P}(\n_0\backslash \{n\})$. We proceed by cases according to the value of $\m_{n-1}$.\\

\noindent\textsc{Case:} When $\m_{n-1}=1$ the results follows easily. Indeed, if $n-1\not\in {\rm Supp}(I)$ the result is direct, and when $n-1\in {\rm Supp}(I)$ we have \\
\begin{eqnarray*}
  \vartheta_n(T_{n-1}XT_{n-1}^{-1})&=& \vartheta_n(T_{n-1}\wmenos EF_IT_{n-1}^{-1}) \\
   &=& \vartheta_n(\wmenos EF_{s_{n-1}(I)} \cancel{T_{n-1}T_{n-1}^{-1}})\\
&=& \wmenos EF_{(s_{n-1}(I))\backslash n}=\wmenos EF_{I\backslash n-1}= \vartheta_{n-1}(X)
\end{eqnarray*}
and the right side follows analogously.\\

\noindent\textsc{Case:} When $\m_{n-1}=B_{n-1}$, we have
\begin{eqnarray*}
  \vartheta_n(T_{n-1}\wmenos B_{n-1}EF_IT_{n-1}^{-1}) &=&  \vartheta_n(\wmenos B_{n}EF_{s_{n-1}(I)}) \\
\end{eqnarray*}
Then when $n-1\in {\rm Supp}(I)$ we obtain
\begin{eqnarray*}
 \vartheta_n(\wmenos B_{n}EF_{s_{n-1}(I)})&=&\W \wmenos EF_{s_{n-1}(I)\backslash n}\\
&=& \W \wmenos EF_{I\backslash n-1}= \vartheta_{n-1}(X)
\end{eqnarray*}
the opposite case follows analogously.\smallbreak

\noindent\textsc{Case:} When $\m_{n-1}=\mathbb{T}_{n-1,k}^{\pm}$ with
$k<n-1$, we have
\begin{eqnarray*}
 \vartheta_n(T_{n-1}XT_{n-1}^{-1}) &=& \vartheta_n(\wmenos T_{n-1} \mathbb{T}_{n-1,k}^{\pm} EF_{I} T_{n-1}^{-1}) \\
   &=&\vartheta_n(\wmenos \mathbb{T}_{n,k}^{\pm} T_{n-1}^{-1} EF_{s_{n-1}(I)})  \\
   &=& \vartheta_n(\wmenos T_{n-2}^{-1}\mathbb{T}_{n,k}^{\pm} EF_{s_{n-1}(I)})  \\
   &=& \z\ \wmenos T_{n-2}^{-1} \mathbb{T}_{n-1,k}^{\pm} EF_{\tau_{n,k}(s_{n-1}(I))}\\
 &=&\z\ \wmenos \mathbb{T}_{n-2,k}^{\pm} EF_{\tau_{n,k}(s_{n-1}(I))}
\end{eqnarray*}
It is not difficult to prove that independently if $n-1$ belong to ${\rm Supp}(I)$ or not, we have that $\tau_{n,k}(s_{n-1}(I))=\tau_{n-1,k}(I)$, thus we obtain
$$\z\ \wmenos \mathbb{T}_{n-2,k}^{\pm} EF_{\tau_{n-1,k}(I)}=\vartheta_{n-1}(X)$$

Finally we have
$$T_{n-1}XT_{n-1}^{-1}=T_{n-1}^{-1}XT_{n-1}+(\U-\U^{-1})(E_{n-1}XT_{n-1}-T_{n-1}XE_{n-1})$$
for $X\in \E$, for details see \cite[Lemma 8]{fjl}. Then applying this relation and (i) Lemma~\ref{commE} we obtain that
$$\vartheta_n(T_{n-1}XT_{n-1}^{-1})=\vartheta_n(T_{n-1}^{-1}XT_{n-1})$$
and (ii) follows.
\end{proof}

\begin{lemma}\label{commT}
  For all $X\in \E$, we have
\begin{equation}\label{commTeq}
  \vartheta_{n-1}(\vartheta_n(XT_{n-1}))=\vartheta_{n-1}(\vartheta_n(T_{n-1}X))
\end{equation}
\end{lemma}
\begin{proof}
  First note that the Eq. (\ref{commTeq}) is equivalent to
\begin{equation}\label{commTalt}
  \vartheta_{n-1}(\vartheta_n(XT_{n-1}^{-1}))=\vartheta_{n-1}(\vartheta_n(T_{n-1}^{-1}X))
\end{equation}

which can be obtained using (ii) Lemma~\ref{commE} and the formula for the inverse, cf. \cite{chpoIMRN}. Then, sometimes we will prove this assertion instead of (\ref{commTeq}) according to its difficulty.\smallbreak
As always we consider $X=\w\m_nEF_I$, and we will distinguish cases according to the possibilities of $\m_n$ and $\m_{n-1}$. We omit the case $\m_n=\m_{n-1}=1$ since it is straightforward.\\

\noindent\textsc{Case:} $\m_n=1$, $\m_{n-1}=\mathbb{T}_{n-1,k}^{\pm}$ with $k<n-1$
\begin{eqnarray*}
  \vartheta_{n-1}(\vartheta_n(T_{n-1}X)) &=&\vartheta_{n-1}(\vartheta_n(T_{n-1}\wmenos\mathbb{T}_{n-1,k}^{\pm} EF_I)) \\
   &=& \vartheta_{n-1}(\vartheta_n(\wmenos T_{n-1}\mathbb{T}_{n-1,k}^{\pm} EF_I))\\
   &=&\wmenos \vartheta_{n-1}(\vartheta_n( \mathbb{T}_{n,k}^{\pm} EF_I))  \\
   &=&\z \wmenos \vartheta_{n-1}(\mathbb{T}_{n-1,k}^{\pm} EF_{\tau_{n,k}(I)})\\
   &=&\z^2 \wmenos \mathbb{T}_{n-2,k}^{\pm} EF_{\tau_{n-1,k}(\tau_{n,k}(I))}
\end{eqnarray*}
On the other hand
\begin{eqnarray*}
  \vartheta_{n-1}(\vartheta_n(XT_{n-1})) &=& \vartheta_{n-1}(\vartheta_n(\wmenos\mathbb{T}_{n-1,k}^{\pm} EF_IT_{n-1})) \\
   &=& \vartheta_{n-1}(\vartheta_n(\wmenos\mathbb{T}_{n-1,k}^{\pm}T_{n-1} EF_{s_{n-1}(I)})) \\
   &=& \z \wmenos\vartheta_{n-1}(\mathbb{T}_{n-1,k}^{\pm} EF_{\tau_{n,n-1}(s_{n-1}(I))}) \\
   &=& \z^2 \wmenos\mathbb{T}_{n-2,k}^{\pm} EF_{\tau_{n-1,k}(\tau_{n,n-1}(s_{n-1}(I)))}
\end{eqnarray*}
Now, note that $\tau_{n,n-1}(s_{n-1}(I))=\tau_{n,n-1}(I)$, then, it is clear that $\tau_{n-1,k}(\tau_{n,n-1}(s_{n-1}(I)))=\tau_{n-1,k}(\tau_{n,k}(I))$.\\

\noindent\textsc{Case:} $\m_n=1$, $\m_{n-1}=B_{n-1}$
\begin{eqnarray*}
  \vartheta_{n-1}(\vartheta_n(T_{n-1}X)) &=&\vartheta_{n-1}(\vartheta_n(T_{n-1}\wmenos B_{n-1} EF_I)) \\
   &=& \vartheta_{n-1}(\vartheta_n(\wmenos T_{n-1}B_{n-1} EF_I))\\
   &=&\z \wmenos \vartheta_{n-1}(B_{n-1} EF_{\tau_{n,n-1}(I)})
 \end{eqnarray*}
On the other hand
\begin{eqnarray*}
  \vartheta_{n-1}(\vartheta_n(X T_{n-1})) &=&\vartheta_{n-1}(\vartheta_n(\wmenos B_{n-1} EF_I T_{n-1})) \\
   &=& \vartheta_{n-1}(\vartheta_n(\wmenos B_{n-1} T_{n-1} EF_{s_{n-1}(I)}))\\
   &=&\z \wmenos \vartheta_{n-1}( B_{n-1} EF_{\tau_{n,n-1}(s_{n-1}(I))})
 \end{eqnarray*}
and we know by the last case that the partitions involved are equal.\\

\noindent\textsc{Case:} $\m_n=\mathbb{T}_{n,k}^{\pm}$ with $k<n$. In this case we will prove (\ref{commTalt}). First suppose that $n\not\in {\rm Supp}(I)$
\begin{eqnarray*}
 \vartheta_{n-1}(\vartheta_n(T_{n-1}^{-1}X))  &=& \vartheta_{n-1}(\vartheta_n(T_{n-1}^{-1}\w\mathbb{T}_{n,k}^{\pm}EF_I))\\
   &=& \vartheta_{n-1}(\vartheta_n(T_{n-1}^{-1}\w T_{n-1}\mathbb{T}_{n-1,k}^{\pm}EF_I))\\
   &=& \vartheta_{n-1}(\vartheta_n(T_{n-1}^{-1}\w T_{n-1})\mathbb{T}_{n-1,k}^{\pm}EF_I) \quad\text{(by Lemma~\ref{conjugation})}\\
   &=& \vartheta_{n-1}(\vartheta_{n-1}(\w)\mathbb{T}_{n-1,k}^{\pm}EF_I)=\vartheta_{n-1}(\w) \vartheta_{n-1}(\mathbb{T}_{n-1,k}^{\pm}EF_I)
\end{eqnarray*}
On the other hand

\begin{eqnarray*}
 \vartheta_{n-1}(\vartheta_n(X T_{n-1}^{-1}))  &=& \vartheta_{n-1}(\vartheta_n(\w\mathbb{T}_{n,k}^{\pm}EF_I T_{n-1}^{-1}))\\
   &=& \vartheta_{n-1}(\vartheta_n(\w T_{n-1}\mathbb{T}_{n-1,k}^{\pm}EF_IT_{n-1}^{-1}))\\
   &=&\vartheta_{n-1}(\w\vartheta_n( T_{n-1}(\mathbb{T}_{n-1,k}^{\pm}EF_I)T_{n-1}^{-1})\quad\text{(by Lemma~\ref{conjugation})} \\
   &=& \vartheta_{n-1}(\w\vartheta_{n-1}( \mathbb{T}_{n-1,k}^{\pm}EF_I))=\vartheta_{n-1}(\w) \vartheta_{n-1}(\mathbb{T}_{n-1,k}^{\pm}EF_I)
\end{eqnarray*}
From now on we suppose that $n\in {\rm Supp}(I)$.\smallbreak

\noindent $\bullet$\textsc{Subcase: }$\m_{n-1}=1$.Fisrt note, that for $k=n-1$ the result follows easily. Then, we can suppose $k<n-1$
\begin{eqnarray*}
 \vartheta_{n-1}(\vartheta_n(T_{n-1}^{-1}X))  &=& \vartheta_{n-1}(\vartheta_n(T_{n-1}^{-1}\wmenos\mathbb{T}_{n,k}^{\pm}EF_I))\\
   &=& \vartheta_{n-1}(\vartheta_n(\wmenos \mathbb{T}_{n-1,k}^{\pm}EF_I))\\
   &=& \x\ \wmenos\vartheta_{n-1}( \mathbb{T}_{n-1,k}^{\pm}EF_{I\backslash n})  \\
   &=& \z\x\ \wmenos \mathbb{T}_{n-2,k}^{\pm}EF_{\tau_{n-1,k}(I\backslash n)}
\end{eqnarray*}
On the other hand
\begin{eqnarray*}
   \vartheta_{n-1}(\vartheta_n(X T_{n-1}^{-1})) &=& \vartheta_{n-1}(\vartheta_n(\wmenos\mathbb{T}_{n,k}^{\pm}EF_I  T_{n-1}^{-1})) \\
   &=& \vartheta_{n-1}(\vartheta_n(\wmenos\mathbb{T}_{n,k}^{\pm}T_{n-1}^{-1}EF_{s_{n-1}(I)}))  \\
   &=&  \vartheta_{n-1}(\vartheta_n(\wmenos T_{n-2}^{-1}\mathbb{T}_{n,k}^{\pm}EF_{s_{n-1}(I)}))  \\
   &=&  \wmenos\vartheta_{n-1}( T_{n-2}^{-1}\vartheta_n(\mathbb{T}_{n,k}^{\pm}EF_{s_{n-1}(I)})) \\
 &=&\z \wmenos\vartheta_{n-1}( T_{n-2}^{-1}\mathbb{T}_{n-1,k}^{\pm}EF_{\tau_{n,k}(s_{n-1}(I))})\\
&=& \z\x \wmenos \mathbb{T}_{n-2,k}^{\pm}EF_{\tau_{n,k}(s_{n-1}(I))\backslash n-1}
\end{eqnarray*}

and it is easy verify that $\tau_{n,k}(s_{n-1}(I))\backslash n-1=\tau_{n-1,k}(I\backslash n)$.\smallbreak

\noindent $\bullet$\textsc{Subcase:} $\m_{n-1}=B_{n-1}$. First suppose $k=n-1$ for the negative case, that is $\m_n=T_{n-1}B_{n-1}$, we have
\begin{eqnarray*}
   \vartheta_{n-1}(\vartheta_n(T_{n-1}^{-1}X)) &=& \vartheta_{n-1}(\vartheta_n(T_{n-1}^{-1}\wmenos B_{n-1}T_{n-1}B_{n-1}EF_I))  \\
   &=& \vartheta_{n-1}(\vartheta_n(\wmenos \cancel{T_{n-1}^{-1} T_{n-1}}B_{n-1}B_n EF_I)) \quad\text{(by (i) Lemma~\ref{restantes} )}  \\
   &=& \wmenos\vartheta_{n-1}(\vartheta_n( B_{n-1}B_n EF_I))
\end{eqnarray*}
On the other hand
\begin{eqnarray*}
\vartheta_{n-1}(\vartheta_n(X T_{n-1}^{-1}))   &=& \vartheta_{n-1}(\vartheta_n(\wmenos B_{n-1}T_{n-1}B_{n-1}EF_I T_{n-1}^{-1})) \\
   &=& \wmenos\vartheta_{n-1}(\vartheta_n( B_{n-1}T_{n-1}B_{n-1} T_{n-1}^{-1}EF_{s_{n-1}(I)})) \\
   &=& \wmenos\vartheta_{n-1}(\vartheta_n( B_{n-1}B_n EF_{s_{n-1}(I)}))
\end{eqnarray*}
If we fix $I_1=I$ and $I_2=s_{n-1}(I)$,
\begin{eqnarray*}
 \vartheta_{n-1}(\vartheta_n( B_{n-1}B_n EF_{I_i}))  &=&  \W \vartheta_{n-1}( B_{n-1}EF_{I_i\backslash n}))  \\
  &=& \W \vartheta_{n-1}( B_{n-1}EF_{(I_i\backslash n)\backslash n-1}) \\
&=&\W^2 EF_{(I_i\backslash n)\backslash n-1}
\end{eqnarray*}
the result follows easily by comparing the partitions for $i=1,2$. In the present case, we are supposing that $n-1\in {\rm Supp}(I)$, for the opposite case we can proceed analogously, and we will obtain the same partitions, but this time, it will appear the parameter $\W\x$ for $i=1,2$ in the final result. We omit the proof for $\m_n=T_{n-1}$ (positive case) since can be verified analogously.\\

For $k<n-1$ we have
\begin{eqnarray*}
  \vartheta_{n-1}(\vartheta_n(T_{n-1}^{-1}X)) &=& \vartheta_{n-1}(\vartheta_n(T_{n-1}^{-1}\wmenos B_{n-1}\mathbb{T}_{n,k}^{\pm}EF_I)) \\
   &=& \wmenos \vartheta_{n-1}(\vartheta_n(T_{n-1}^{-1}B_{n-1}T_{n-1}T_{n-2}\mathbb{T}_{n-2,k}^{\pm}EF_I))\\
   &=&\wmenos \underbrace{\vartheta_{n-1}(\vartheta_n(T_{n-1}^{-1}B_{n-1}T_{n-1}T_{n-2}EF_{\sigma(I)}))}_{\mathtt A}\mathbb{T}_{n-2,k}^{\pm} \\
\end{eqnarray*}
where $\sigma=\sigma_{n-2,k}$. Now, let's compute ${\mathtt A}$
\begin{eqnarray*}
  {\mathtt A} &=& \vartheta_{n-1}(\vartheta_n(T_{n-1}^{-1}B_{n-1}T_{n-1}T_{n-2}EF_{\sigma(I)})) \\
   &=& \vartheta_{n-1}(\vartheta_n(T_{n-1}B_{n-1}T_{n-1}T_{n-2}EF_{\sigma(I)}))-(\U-\U^{-1})\vartheta_{n-1}(\vartheta_n(E_{n-1}B_{n-1}T_{n-1}T_{n-2}EF_{\sigma(I)})) \\
   &=& \vartheta_{n-1}(\vartheta_n(T_{n-1}B_{n-1}T_{n-1}^{-1}T_{n-2}EF_{\sigma(I)}))+(\U-\U^{-1})\vartheta_{n-1}(\vartheta_n(T_{n-1}B_{n-1}E_{n-1}T_{n-2}EF_{\sigma(I)}))\\
&&-(\U-\U^{-1})\vartheta_{n-1}(\vartheta_n(E_{n-1}B_{n-1}T_{n-1}T_{n-2}EF_{\sigma(I)})) \\
   &=& \vartheta_{n-1}(\vartheta_n(T_{n-2}B_{n}EF_{\sigma(I)}))+(\U-\U^{-1})\vartheta_{n-1}(\vartheta_n(T_{n-1}B_{n-1}T_{n-2}EF_{\sigma(I)*\{n,n-2\}}))\\
&&-(\U-\U^{-1})\vartheta_{n-1}(\vartheta_n(B_{n-1}T_{n-1}T_{n-2}EF_{\sigma(I)*\{n,n-2\}})) \\
&=&\W \vartheta_{n-1}(T_{n-2}EF_{\sigma(I)\backslash n})+(\U-\U^{-1})\vartheta_{n-1}(\vartheta_n(T_{n-1}T_{n-2}B_{n-2}EF_{\sigma(I)*\{n,n-2\}}))\\
&&-\z(\U-\U^{-1})\vartheta_{n-1}(B_{n-1}T_{n-2}EF_{\tau_{n,n-2}(\sigma(I)*\{n,n-2\})}) \\
&=&\W \z EF_{\tau_{n-1,n-2}(\sigma(I)\backslash n)}+\z(\U-\U^{-1})\vartheta_{n-1}(T_{n-2}B_{n-2}EF_{\tau_{n,n-2}(\sigma(I)*\{n,n-2\})})\\
&&-\z(\U-\U^{-1})\vartheta_{n-1}(T_{n-2}B_{n-2}EF_{\tau_{n,n-2}(\sigma(I)*\{n,n-2\})}) \\
&=&\W \z EF_{\tau_{n-1,n-2}(\sigma(I)\backslash n)}
\end{eqnarray*}

On the other hand

\begin{eqnarray*}
   \vartheta_{n-1}(\vartheta_n(X T_{n-1}^{-1}))&=& \vartheta_{n-1}(\vartheta_n(\wmenos B_{n-1}\mathbb{T}_{n,k}^{\pm}EF_I T_{n-1}^{-1})) \\
   &=&  \wmenos\vartheta_{n-1}(\vartheta_n(B_{n-1}T_{n-2}^{-1}\mathbb{T}_{n,k}^{\pm}EF_J))\\
   &=& \wmenos \underbrace{\vartheta_{n-1}(\vartheta_n(B_{n-1}T_{n-2}^{-1}T_{n-1}T_{n-2}EF_{\sigma(J)}))}_{\mathtt D}\mathbb{T}_{n-2,k}^{\pm}
 \end{eqnarray*}

where $J=s_{n-1}(I)$ and $\sigma=\sigma_{n-2,k}$. Now, we compute ${\mathtt D}$
\begin{eqnarray*}
  {\mathtt D} &=& \vartheta_{n-1}(\vartheta_n(B_{n-1}T_{n-2}^{-1}T_{n-1}T_{n-2}EF_{\sigma(J)})) \\
   &=&  \vartheta_{n-1}(\vartheta_n(B_{n-1}T_{n-2}T_{n-1}T_{n-2}EF_{\sigma(J)}))-(\U-\U^{-1})\vartheta_{n-1}(\vartheta_n(B_{n-1}E_{n-2}T_{n-1}T_{n-2}EF_{\sigma(J)}))\\
   &=& \vartheta_{n-1}(\vartheta_n(T_{n-2}B_{n-2}T_{n-1}T_{n-2}EF_{\sigma(J)}))-(\U-\U^{-1})\vartheta_{n-1}(\vartheta_n(B_{n-1}T_{n-1}T_{n-2}EF_{\sigma(J)*\{n,n-1\}})) \\
   &=& \z\vartheta_{n-1}(T_{n-2}B_{n-2}T_{n-2}EF_{\tau_{n,n-2}(\sigma(J))})-\z(\U-\U^{-1})\vartheta_{n-1}(B_{n-1}T_{n-2}EF_{\tau_{n,n-2}(\sigma(J)*\{n,n-1\})}) \\
   &=&   \z\vartheta_{n-1}(B_{n-1}EF_{\tau_{n,n-2}(\sigma(J))})+ \z(\U-\U^{-1})\vartheta_{n-1}(T_{n-2}B_{n-2}EF_{(\tau_{n,n-2}(\sigma(J)))*\{n-1,n-2\}})\\
   &&-\z(\U-\U^{-1})\vartheta_{n-1}(T_{n-2}B_{n-2}EF_{\tau_{n,n-2}(\sigma(J)*\{n,n-1\})}) \\
   &=&    \z\W EF_{\tau_{n,n-2}(\sigma(J))\backslash n-1}
\end{eqnarray*}
and it is not difficult verify that the partitions involved coincide.\smallbreak

\noindent $\bullet$\textsc{Subcase:} $\m_{n-1}=\mathbb{T}_{n-1,j}^{\pm}$. We will distinguish 2 subcases\\

First suppose that $k<n-1$
\begin{eqnarray*}
  \vartheta_{n-1}(\vartheta_n(T_{n-1}^{-1}X)) &=& \vartheta_{n-1}(\vartheta_n(T_{n-1}^{-1}\wmenos\mathbb{T}_{n-1,j}^{\pm}\mathbb{T}_{n,k}^{\pm}EF_I))\\
&=& \wmenos \vartheta_{n-1}(\vartheta_n(T_{n-1}^{-1}T_{n-2}\mathbb{T}_{n-2,j}^{\pm}T_{n-1}\mathbb{T}_{n-1,k}^{\pm}EF_I))\\
&=&\wmenos\vartheta_{n-1}(\vartheta_n(T_{n-1}T_{n-2}T_{n-2}^{-1}\mathbb{T}_{n-2,j}^{\pm}\mathbb{T}_{n-1,k}^{\pm}EF_I))  \\
&=&\wmenos \vartheta_{n-1}(\vartheta_n(T_{n-2}T_{n-1}T_{n-2}^{-1}\mathbb{T}_{n-2,j}^{\pm}\mathbb{T}_{n-1,k}^{\pm}EF_I))  \\
&=& \wmenos\vartheta_{n-1}(T_{n-2}\vartheta_n(T_{n-1}T_{n-2}^{-1}EF_{\varphi(I)})\mathbb{T}_{n-2,j}^{\pm}\mathbb{T}_{n-1,k}^{\pm}) \\
\end{eqnarray*}
where $\varphi=\sigma_{n-2,j}\sigma_{n-1,k}$. Further it is easy to prove that $$\vartheta_n(T_{n-1}T_{n-2}^{-1}EF_{\varphi(I}))=\z\ T_{n-2}^{-1}EF_{\tau_{n,n-2}(\varphi(I))}$$
then we obtain the following. First we consider $j<k$.
\begin{eqnarray*}
   \vartheta_{n-1}(\vartheta_n(T_{n-1}^{-1}X))  &=& \z\wmenos\vartheta_{n-1}(\cancel{T_{n-2}T_{n-2}^{-1}}EF_{\tau_{n,n-2}(\varphi(I))}\mathbb{T}_{n-2,j}^{\pm}\mathbb{T}_{n-1,k}^{\pm}) \\
   &=&\z\wmenos\vartheta_{n-1}(\mathbb{T}_{n-2,j}^{\pm}\mathbb{T}_{n-1,k}^{\pm}EF_{\tau_{n,j}(I)})  \\
   &=& \z^2\wmenos\mathbb{T}_{n-2,j}^{\pm}\mathbb{T}_{n-2,k}^{\pm}EF_{\tau_{n-1,k}(\tau_{n,j}(I))}
\end{eqnarray*}

On the other hand
\begin{eqnarray*}
 \vartheta_{n-1}(\vartheta_n(X T_{n-1}^{-1}))  &=&  \vartheta_{n-1}(\vartheta_n(\wmenos\mathbb{T}_{n-1,j}^{\pm}\mathbb{T}_{n,k}^{\pm}EF_I T_{n-1}^{-1})) \\
   &=& \wmenos \vartheta_{n-1}(\vartheta_n(\mathbb{T}_{n-1,j}^{\pm}\mathbb{T}_{n,k}^{\pm}T_{n-1}^{-1}EF_J)),\quad \text{where $J=s_{n-1}(I)$}\\
   &=&\wmenos  \vartheta_{n-1}(\mathbb{T}_{n-1,j}^{\pm}T_{n-2}^{-1}\vartheta_n(\mathbb{T}_{n,k}^{\pm}EF_J))\\
   &=& \z \wmenos\vartheta_{n-1}(\mathbb{T}_{n-1,j}^{\pm}T_{n-2}^{-1}\mathbb{T}_{n-1,k}^{\pm}EF_{\tau_{n,k}(J)})\\
&=& \z\wmenos \vartheta_{n-1}(\mathbb{T}_{n-1,j}^{\pm}\mathbb{T}_{n-2,k}^{\pm}EF_{\tau_{n,k}(J)} \\
   &=& \z \wmenos\vartheta_{n-1}(\mathbb{T}_{n-1,j}^{\pm}EF_{\sigma_{n-2,k}(\tau_{n,k}(J))})\mathbb{T}_{n-2,k}^{\pm} \\
   &=& \z^2\wmenos \mathbb{T}_{n-2,j}^{\pm}EF_{\tau_{n-1,j}(\sigma_{n-2,k}(\tau_{n,k}(J)))}\mathbb{T}_{n-2,k}^{\pm} \\
   &=&\z^2 \wmenos\mathbb{T}_{n-2,j}^{\pm}T_{n-2}^{-1}\mathbb{T}_{n-1,k}^{\pm}EF_{\tau_{n-1,j}(\tau_{n,k}(J))}  \\
 \end{eqnarray*}
and the result follows by comparing the partitions. Note that for $j\geq k$ the proof is the same but will appear the term $j+1$ instead of $j$, in both partitions. Finally suppose $k=n-1$, we only prove the negative case, that is $\m_n=T_{n-1}B_{n-1}$, since for the positive case we can proceed analogously, and it is easier
\begin{eqnarray*}
  \vartheta_{n-1}(\vartheta_n(T_{n-1}^{-1}X)) &=& \vartheta_{n-1}(\vartheta_n(T_{n-1}^{-1}\wmenos\mathbb{T}_{n-1,j}^{\pm}T_{n-1}B_{n-1}EF_I)) \\
   &=&\wmenos \vartheta_{n-1}(\vartheta_n(T_{n-1}^{-1}T_{n-2}\mathbb{T}_{n-2,j}^{\pm}T_{n-1}B_{n-1}EF_I)) \\
   &=& \wmenos \vartheta_{n-1}(\vartheta_n(T_{n-1}^{-1}T_{n-2}T_{n-1}\mathbb{T}_{n-2,j}^{\pm}B_{n-1}EF_I))\\
   &=&\wmenos  \vartheta_{n-1}(\vartheta_n(T_{n-2}T_{n-1}T_{n-2}^{-1}\mathbb{T}_{n-2,j}^{\pm}B_{n-1}EF_I))\\
   &=& \wmenos\vartheta_{n-1}(T_{n-2}\vartheta_n(T_{n-1}T_{n-2}^{-1}EF_{\sigma_{n-2,j}(I)})\mathbb{T}_{n-2,j}^{\pm}B_{n-1} )\\
   &=&  \z\wmenos\vartheta_{n-1}(\cancel{T_{n-2}T_{n-2}^{-1}}EF_{\tau_{n,n-2}(\sigma_{n-2,j}(I))}\mathbb{T}_{n-2,j}^{\pm}B_{n-1} )\\
   &=&   \z\wmenos\vartheta_{n-1}(\mathbb{T}_{n-2,j}^{\pm}B_{n-1}EF_{\tau_{n,j}(I)} )
\end{eqnarray*}
now, depending if $n-1\in {\rm Supp}(I)$ or not, we can obtain
$$\z\W \wmenos\mathbb{T}_{n-2,j}^{\pm}EF_{\tau_{n,j}(I)\backslash n-1}\quad \text{or} \quad \z\y \wmenos\mathbb{T}_{n-2,j}^{\pm}EF_{\tau_{n,j}(I)}$$
respectively.\\

On the other hand
\begin{eqnarray*}
   \vartheta_{n-1}(\vartheta_n(X T_{n-1}^{-1})) &=&  \vartheta_{n-1}(\vartheta_n(\wmenos\mathbb{T}_{n-1,j}^{\pm}T_{n-1}B_{n-1}EF_I T_{n-1}^{-1})) \\
   &=& \wmenos \vartheta_{n-1}(\vartheta_n(\mathbb{T}_{n-1,j}^{\pm}T_{n-1}B_{n-1}T_{n-1}^{-1}EF_J))\\
   &=&\wmenos \vartheta_{n-1}(\vartheta_n(\mathbb{T}_{n-1,j}^{\pm}B_nEF_J))=C
\end{eqnarray*}
Now depending of $n-1\in {\rm Supp}(I)$ or not, we obtain
\begin{equation*}
  \begin{array}{rcl|rcl}
 C  &=&\W \wmenos\vartheta_{n-1}(\mathbb{T}_{n-1,j}^{\pm}EF_{J\backslash n})  &   C& = & \y \wmenos\vartheta_{n-1}(\mathbb{T}_{n-1,j}^{\pm}EF_{J}) \\
   & = & \z\W\wmenos\mathbb{T}_{n-2,j}^{\pm}EF_{\tau_{n-1,j}(J\backslash n)} &  & = &  \z\y\wmenos\mathbb{T}_{n-2,j}^{\pm}EF_{\tau_{n-1,j}(J\backslash n)}
\end{array}
\end{equation*}
respectively. And, it is easy verify that the partitions are the same, therefore this case follows.\smallbreak

\noindent\textsc{Case:} $\m_{n}=B_{n}$, $\m_{n-1}=1$
\begin{eqnarray*}
  \vartheta_{n-1}(\vartheta_n(X T_{n-1})) &=&\vartheta_{n-1}(\vartheta_n(\wmenos B_nEF_I T_{n-1}))  \\
   &=& \vartheta_{n-1}(\vartheta_n(\wmenos B_n T_{n-1}EF_{s_{n-1}(I)}))  \\
   &=& \vartheta_{n-1}(\vartheta_n(\wmenos T_{n-1}B_{n-1}EF_{s_{n-1}(I)})) \\
   &=&\z\wmenos\vartheta_{n-1}(B_{n-1}EF_{\tau_{n,n-1}(s_{n-1}(I))})  \\
   &=& \z\W\wmenos EF_{\tau_{n,n-1}(s_{n-1}(I))\backslash n-1}
\end{eqnarray*}
On the other hand
\begin{eqnarray*}
   \vartheta_{n-1}(\vartheta_n(T_{n-1}X)) &=&\vartheta_{n-1}(\vartheta_n(T_{n-1}\wmenos B_nEF_I))  \\
   &=&\wmenos\vartheta_{n-1}(\vartheta_n(T_{n-1} B_nEF_I)) \\
   &=& \wmenos \vartheta_{n-1}(\vartheta_n(T_{n-1}^2B_{n-1}T_{n-1}^{-1}EF_I)) \\
\end{eqnarray*}
expanding the square and the inverse we have that

$$\vartheta_{n-1}(\vartheta_n(T_{n-1}^2B_{n-1}T_{n-1}^{-1}EF_I))=A-(\U-\U^{1})B+(\U-\U^{1})C$$
where
\begin{eqnarray*}
  A &:=& \vartheta_{n-1}(\vartheta_n(B_{n-1}T_{n-1}EF_I)) \\
  B &:=&  \vartheta_{n-1}(\vartheta_n(B_{n-1}E_{n-1}EF_I))\\
  C &:=&  \vartheta_{n-1}(\vartheta_n(E_{n-1}B_{n}EF_I))
\end{eqnarray*}
Now, by direct computations we have that
\begin{eqnarray*}
  A &=& \z \vartheta_{n-1}(B_{n-1}EF_{\tau_{n,n-1}(I)}) \\
   &=& \z\w EF_{\tau_{n,n-1}(I)\backslash n-1} \\
\end{eqnarray*}
\begin{eqnarray*}
  B &=& \vartheta_{n-1}(\vartheta_n(B_{n-1}EF_{I*\{n-1,n\}})) \\
   &=& \x\vartheta_{n-1}(B_{n-1}EF_{(I*\{n-1,n\})\backslash n}) \\
   &=& \x\W B_{n-1}EF_{((I*\{n-1,n\})\backslash n)\backslash n-1}
\end{eqnarray*}
and
\begin{eqnarray*}
  C &=& \vartheta_{n-1}(\vartheta_n(B_{n}EF_{I*\{n-1,n\}})) \\
   &=& \W\vartheta_{n-1}(EF_{I*\{n-1,n\}\backslash n}) \\
   &=& \x \W\vartheta_{n-1}(EF_{(I*\{n-1,n\}\backslash n)\backslash n-1}
\end{eqnarray*}
clearly we have that $B=C$ and also that $\tau_{n,n-1}(I)\backslash n-1=\tau_{n,n-1}(s_{n-1}(I))\backslash n-1$, thus the results follows.\smallbreak
\noindent\textsc{Case:} $\m_n=B_n$, $\m_{n-1}=B_{n-1}$
\begin{eqnarray*}
 \vartheta_{n-1}(\vartheta_n(X T_{n-1}))  &=& \vartheta_{n-1}(\vartheta_n(\wmenos B_{n-1}B_n EF_I T_{n-1}))  \\
   &=&  \vartheta_{n-1}(\vartheta_n(\wmenos B_{n-1}B_n T_{n-1}EF_{s_{n-1}(I)}))  \\
   &=&\wmenos \vartheta_{n-1}(\vartheta_n(B_{n-1}T_{n-1}B_{n-1}EF_{s_{n-1}(I)})) \\
\end{eqnarray*}
On the other hand
\begin{eqnarray*}
 \vartheta_{n-1}(\vartheta_n(T_{n-1}X))  &=& \vartheta_{n-1}(\vartheta_n(T_{n-1}\wmenos B_{n-1}B_{n}EF_I))\\
   &=& \wmenos\vartheta_{n-1}(\vartheta_n(T_{n-1} B_{n-1}B_{n}EF_I)) \\
   &=& \wmenos\vartheta_{n-1}(\vartheta_n(B_{n-1}T_{n-1}B_{n-1}EF_I))
\end{eqnarray*}
by using (i) Lemma~\ref{restantes}. Now denote $I_1=I$ and $I_2=s_{n-1}(I)$. Then we have
\begin{eqnarray*}
  \vartheta_{n-1}(\vartheta_n(B_{n-1}T_{n-1}B_{n-1}EF_{I_i})) &=&  \vartheta_{n-1}(B_{n-1}\vartheta_n(T_{n-1}EF_{I_i})B_{n-1}) \\
   &=& \z\vartheta_{n-1}(B_{n-1}EF_{\tau_{n,n-1}(I_i)}B_{n-1})  \\
   &=&  \z\vartheta_{n-1}(B_{n-1}^2EF_{\tau_{n,n-1}(I_i)})\\
   &=&  \z[\vartheta_{n-1}(EF_{\tau_{n,n-1}(I_i)})+ (\V-\V^{-1}) \vartheta_{n-1}(B_{n-1}EF_{\tau_{n,n-1}(I_i)*\{0,n-1\}}) ]\\
&=& \z[\x EF_{\tau_{n,n-1}(I_i)\backslash n-1}+ \W (\V-\V^{-1})\vartheta_{n-1}(B_{n-1}EF_{(\tau_{n,n-1}(I_i)*\{0,n-1\})\backslash n-1}) ]
\end{eqnarray*}
finally it is not difficult verify that are the same for $i=1,2$.\\

\noindent\textsc{Case:} $\m_n=B_n$, $\m_{n-1}=\mathbb{T}_{n-1,k}^{\pm}$ with $k<n-1$. We proceed first with the positive case
\begin{eqnarray*}
\vartheta_{n-1}(\vartheta_n(X T_{n-1}))   &=& \vartheta_{n-1}(\vartheta_n(\wmenos\mathbb{T}_{n-1,k}^{+} B_n EF_I T_{n-1})) \\
   &=&\wmenos\vartheta_{n-1}(\mathbb{T}_{n-1,k}^{+}\vartheta_n( B_n T_{n-1}EF_{s_{n-1}(I)})  \\
   &=&\wmenos\vartheta_{n-1}(\mathbb{T}_{n-1,k}^{+}\vartheta_n( T_{n-1}B_{n-1}EF_{s_{n-1}(I)})  \\
   &=&\z\ \wmenos \vartheta_{n-1}(\mathbb{T}_{n-1,k}^{+} B_{n-1}EF_{\tau_{n,n-1}(s_{n-1}(I))}) \\
   &=&\z\ \wmenos\vartheta_{n-1}( T_{n-2}B_{n-1}EF_{\sigma(I_1)})\mathbb{T}_{n-2,k}^{+}
\end{eqnarray*}
where $I_1=\tau_{n,n-1}(s_{n-1}(I))$ and $\sigma=\sigma_{n-2,k}$. Now expanding the square and the inverse we obtain that
\begin{eqnarray*}
 \vartheta_{n-1}( T_{n-2}B_{n-1}EF_{\sigma(I_1)})  &=& \vartheta_{n-1}( B_{n-2}T_{n-2}EF_{\sigma(I_1)})-(\U-\U^{-1})\vartheta_{n-1}( B_{n-2}E_{n-2}EF_{\sigma(I_1)})+\\
   &&(\U-\U^{-1})\vartheta_{n-1}( B_{n-1}E_{n-2}EF_{\sigma(I_1)})  \\
   &=& B_{n-2}\vartheta_{n-1}(T_{n-2}EF_{\sigma(I_1)})-(\U-\U^{-1})B_{n-2}\vartheta_{n-1}(EF_{\sigma(I_1)*\{n-1,n-2\}})+\\
   &&(\U-\U^{-1})\vartheta_{n-1}( B_{n-1}EF_{\sigma(I_1)*\{n-1,n-2\}})  \\
   &=& \z B_{n-2}EF_{\tau_{n-1,n-2}(\sigma(I_1))}-\x(\U-\U^{-1})B_{n-2}EF_{(\sigma(I_1)*\{n-1,n-2\})\backslash n-1}+\\
   &&\W(\U-\U^{-1}) EF_{(\sigma(I_1)*\{n-1,n-2\})\backslash n-1}  \\
\end{eqnarray*}
On the other hand
\begin{eqnarray*}
 \vartheta_{n-1}(\vartheta_n(T_{n-1}X))  &=& \vartheta_{n-1}(\vartheta_n(T_{n-1}\wmenos\mathbb{T}_{n-1,k}^{+}B_n EF_I)) \\
   &=&\wmenos\vartheta_{n-1}(\vartheta_n(\mathbb{T}_{n,k}^{+}B_n EF_I)) \\
   &=& \wmenos[\vartheta_{n-1}(\vartheta_n(B_{n-1}T_{n-1}\mathbb{T}_{n-1,k}^{+} EF_I))-\\
&&(\U-\U^{-1})\vartheta_{n-1}(\vartheta_n(B_{n-1}E_{n-1}\mathbb{T}_{n-1,k}^{+} EF_I)) +
    (\U-\U^{-1})\vartheta_{n-1}(\vartheta_n(B_{n}E_{n-1}\mathbb{T}_{n-1,k}^{+} EF_I))]
  \end{eqnarray*}
Let's compute each term separately
 \begin{eqnarray*}
 \bullet  \vartheta_{n-1}(\vartheta_n(B_{n-1}T_{n-1}\mathbb{T}_{n-1,k}^{+} EF_I)) &=& \vartheta_{n-1}(B_{n-1}\vartheta_n(T_{n-1} EF_{\varphi( I)})\mathbb{T}_{n-1,k}^{+})\quad \text{with $\varphi=\sigma_{n-1,k}$} \\
    &=&\z\vartheta_{n-1}(B_{n-1}EF_{\tau_{n,n-1}\varphi(I)}\mathbb{T}_{n-1,k}^{+}) \\
    &=&\z \vartheta_{n-1}(B_{n-1}\mathbb{T}_{n-1,k}^{+}EF_{\tau_{n,k}(I)} \quad \text{(by (iii) Lemma~\ref{partitionsprop})} \\
    &=&\z \vartheta_{n-1}(T_{n-2}B_{n-2}\mathbb{T}_{n-2,k}^{+}EF_{\tau_{n,k}(I)}) \\
&=& \z \vartheta_{n-1}(T_{n-2}B_{n-2}EF_{\sigma(\tau_{n,k}(I))}) \mathbb{T}_{n-2,k}^{+}\\
&=& \z^2 B_{n-2}EF_{\tau_{n-1,n-2}(\sigma(\tau_{n,k}(I)))} \mathbb{T}_{n-2,k}^{+}
 \end{eqnarray*}
\\
\begin{eqnarray*}
 \bullet \vartheta_{n-1}(\vartheta_n(B_{n-1}E_{n-1}\mathbb{T}_{n-1,k}^{+} EF_I)) &=& \vartheta_{n-1}(\vartheta_n(B_{n-1}\mathbb{T}_{n-1,k}^{+}E_{n,k} EF_I)) \\
   &=& \vartheta_{n-1}(\vartheta_n(B_{n-1}\mathbb{T}_{n-1,k}^{+}EF_{I*\{n,k\}})) \\
   &=& \x\vartheta_{n-1}(B_{n-1}\mathbb{T}_{n-1,k}^{+}EF_{(I*\{n,k\})\backslash n}) \\
   &=& \x\vartheta_{n-1}(T_{n-2}B_{n-2}EF_{\sigma(\tau_{n,k}(I))})\mathbb{T}_{n-2,k}^{+} \\
   &=&\z\x B_{n-2}EF_{\tau_{n-1,n-2}(\sigma(\tau_{n,k}(I)))}\mathbb{T}_{n-2,k}^{+}
\end{eqnarray*}

\begin{eqnarray*}
\bullet  \vartheta_{n-1}(\vartheta_n(B_{n}E_{n-1}\mathbb{T}_{n-1,k}^{+} EF_I))  &=&  \vartheta_{n-1}(\vartheta_n(B_{n}\mathbb{T}_{n-1,k}^{+} E_{n,k}EF_I))\\
   &=&  \vartheta_{n-1}(\vartheta_n(\mathbb{T}_{n-1,k}^{+}B_{n}EF_{I*\{n,k\}})) \\
   &=& \W\vartheta_{n-1}(\mathbb{T}_{n-1,k}^{+}EF_{(I*\{n,k\})\backslash n})) \\
   &=& \z\W \mathbb{T}_{n-2,k}^{+}EF_{\tau_{n-1,k}(\tau_{n,k}(I))}\\
&=&\z\W EF_{\tau_{n-1,n-2}(\sigma(\tau_{n,k}(I)))}\mathbb{T}_{n-2,k}^{+}
\end{eqnarray*}
Finally, by using (iii) Lemma~\ref{partitionsprop} we have
\begin{eqnarray*}
  \tau_{n-1,n-2}(\sigma(\tau_{n,k}(I))) &=& \tau_{n-1,n-2}((\tau_{n,n-2}(\sigma(I)))) \\
   &=& (\sigma(I)*\{n,n-1,n-2\})\backslash \{n,n-1\}
\end{eqnarray*}
and, on the other hand
\begin{eqnarray*}
  \tau_{n-1,n-2}(\sigma(I_1)) &=& \tau_{n-1,n-2}(\tau_{n,n-1}(\sigma(s_{n-1}(I))))  \\
   &=& \tau_{n-1,n-2}(\tau_{n,n-1}(s_{n-1}(\sigma(I))))\\
&=& (s_{n-1}(\sigma(I))*\{n,n-1,n-2\})\backslash \{n,n-1\}
\end{eqnarray*}
since $s_{n-1}$ only moves the elements $n$ and $n-1$, and these are removed by the partition, the equality follows.\\
\smallbreak
For the negative case, we have
\begin{eqnarray*}
 \vartheta_{n-1}(\vartheta_n(X T_{n-1}))  &=& \vartheta_{n-1}(\vartheta_n(\wmenos\mathbb{T}_{n-1,k}^{-}B_n T_{n-1}EF_{s_{n-1}(I)})) \\
   &=& \wmenos \vartheta_{n-1}(\vartheta_n(\mathbb{T}_{n-1,k}^{-}T_{n-1}B_{n-1}EF_{s_{n-1}(I)}))\\
   &=&  \z\wmenos\vartheta_{n-1}(\mathbb{T}_{n-1,k}^{-}B_{n-1}EF_{\tau_{n,n-1}(s_{n-1}(I))})\\
   &=& \z\wmenos\vartheta_{n-1}(B_{n-2}\mathbb{T}_{n-1,k}^{-}EF_{\tau_{n,n-1}(s_{n-1}(I))})\quad\text{(by (ii) Lemma~\ref{restantes})} \\
  &=&\z \wmenos B_{n-2}\vartheta_{n-1}(\mathbb{T}_{n-1,k}^{-}EF_{\tau_{n,n-1}(s_{n-1}(I))}) \\
&=&\z^2 \wmenos B_{n-2}\mathbb{T}_{n-2,k}^{-}EF_{\tau_{n-1,k}(\tau_{n,n-1}(s_{n-1}(I)))} \\
&=& \z^2 \wmenos B_{n-2}EF_{\sigma(\tau_{n-1,k}(\tau_{n,n-1}(s_{n-1}(I))))}\mathbb{T}_{n-2,k}^{-}
\end{eqnarray*}
and for the other side
\begin{eqnarray*}
  \vartheta_{n-1}(\vartheta_n(T_{n-1}X)) &=& \vartheta_{n-1}(\vartheta_n(T_{n-1}\wmenos\mathbb{T}_{n-1,k}^{-}B_n EF_{I})) \\
   &=&  \wmenos\vartheta_{n-1}(\vartheta_n(\mathbb{T}_{n,k}^{-}B_n EF_{I}))\\
   &=&\wmenos\vartheta_{n-1}(\vartheta_n(B_{n-1}\mathbb{T}_{n,k}^{-} EF_{I}))  \\
   &=&\z\wmenos \vartheta_{n-1}(B_{n-1}\mathbb{T}_{n-1,k}^{-} EF_{\tau_{n,k}(I)}) \\
   &=& \z \wmenos\vartheta_{n-1}(T_{n-2}B_{n-2} EF_{\sigma(\tau_{n,k}(I))})\mathbb{T}_{n-2,k}^{-} \\
&=&\z^2 \wmenos B_{n-2} EF_{\tau_{n-1,n-2}(\sigma(\tau_{n,k}(I)))}\mathbb{T}_{n-2,k}^{-} \\
&=&\z^2 \wmenos B_{n-2} EF_{\tau_{n-1,n-2}(\sigma(\tau_{n,k}(I)))}\mathbb{T}_{n-2,k}^{-}
\end{eqnarray*}
Finally, using (iii) Lemma~\ref{partitionsprop} we have that
\begin{eqnarray*}
 \sigma(\tau_{n-1,k}(\tau_{n,n-1}(s_{n-1}(I))))  &=& \tau_{n-1,n-2}(\tau_{n,n-1}(\sigma (s_{n-1}(I)))) \\
   &=& [(\sigma (s_{n-1}(I))*\{n,n-1\})\backslash n]*\{n-1,n-2\}\backslash n-1\\
   &=& [(\sigma (s_{n-1}(I))*\{n,n-1,n-2\}]\backslash\{n,n-1\}\\
&=&[(\sigma (I)*\{n,n-1,n-2\}]\backslash\{n,n-1\}\\
&=&\tau_{n-1,n-2}(\tau_{n,n-2}(\sigma (I))=\tau_{n-1,n-2}(\sigma(\tau_{n,k}(I)))
\end{eqnarray*}

\end{proof}

Let be $\mathtt{tr}_n:\E\rightarrow \mathbb{L}$ the linear map defined inductively as follows by, $\mathtt{tr}_1=\vartheta_1$ and
\begin{equation}\label{tracefinal}
\mathtt{tr}_n:=\mathtt{tr}_{n-1}\circ\vartheta_n
\end{equation}
and, let us denote by $\mathtt{tr}$ the family $\{\mathtt{tr}_n\}_{n\geq 1}$. Then, we have the following result
\begin{theorem}\label{Markovtrace}
  $\mathtt{tr}$ is a Markov trace on $\{\E\}_{n\geq 1}$. That is, for every $n\geq 1$ the linear map $\mathtt{tr}_n:\E\rightarrow \mathbb{L}$ satisfies the following properties.

\begin{itemize}
  \item[i)] $\mathtt{tr}_{n}(1)  = 1$
  \item[ii)] $\mathtt{tr}_{n+1}(XT_{n}) = \mathtt{tr}_{n+1}(XE_{n}T_n)= \z \mathtt{tr}_{n}(X) $
  \item[iii)] $\mathtt{tr}_{n+1}(XE_{n}) =\x \mathtt{tr}_{n}(X)$
  \item[iv)] $\mathtt{tr}_{n+1}(XB_{n}) =\y \mathtt{tr}_{n}(X)$
  \item[v)] $\mathtt{tr}_{n+1}(XB_{n}E_n) =\mathtt{tr}_{n+1}(XB_{n}F_{n+1})=\W \mathtt{tr}_{n}(X)$
  \item[vi)] $\mathtt{tr}_n(XY)  =  \mathtt{tr_n}(YX) $
where $X,Y\in \E$
\end{itemize}
for all $n\geq 1$.
\end{theorem}
\begin{proof}
  Rules (ii)--(v) are direct  consequences of  Lemma \ref{lemmatraza1} (ii).  We will prove rule (vi) by induction on $n$. For $n=1$, the rule  holds since $\mathcal{E}_1^{\mathtt B}$ is commutative. Suppose now that (vi) is true for all $k$ less than $n$. For  $Y\in \mathcal{E}_{n-1}^{\mathtt{B}}$ and $X\in \E$ the result follows easily by Lemma~\ref{lemmatraza1} and induction hypothesis, cf. \cite[Theorem 3]{aijuMMJ}. Thus, $ \mathtt{tr}_n(XY) =  \mathtt{tr}_n(YX)$ for all $X\in \E$ and $Y\in \mathcal{E}_{n-1}^{\mathtt{B}}$.\smallbreak
 Further, for   $Y\in \{T_{n-1}, E_{n-1}\}$ we have
$$
\mathtt{tr}_n(XY) = \mathtt{tr}_{n-2}(\mathtt{tr}_{n-1}(\mathtt{tr}_n(XY)))
= \mathtt{tr}_{n-2}(\mathtt{tr}_{n-1}(\mathtt{tr}_n(YX)))
$$
by using Lemmas \ref{commT} and \ref{commE}.\smallbreak
Therefore, we have
$$
\mathtt{tr}_n(XY) = \mathtt{tr}_n(XY)
$$
for all $X\in \E$ and $Y\in \mathcal{E}_{n-1}^{\B} \cup \{T_{n-1}, E_{n-1}\}$, thus, having in mind the linearity of ${\mathtt tr}_n$, the result follows.
\end{proof}

\subsection{} \textbf{Knot invariants from $\E$.} In order to define a new invariant of classical links in the solid torus, we recall some necessary facts. The closure of a braid $\alpha$ in the group $\widetilde{W}_n$ (recall Section~\ref{typeB}), is defined by joining with simple (unknotted and unlinked) arcs its corresponding endpoints, and it is denoted by $\widehat{\alpha}$. The result of closure, $\widehat{\alpha}$, is a link in the solid torus, denoted $ST$. This can be regarded by viewing the closure of the fixed strand as the complementary solid torus. For an example of a link in the solid torus see Figure~\ref{flink}. By the analogue of the Markov theorem for $ST$ (cf. for example \cite{la1,la2}), isotopy classes of oriented links in $ST$ are in bijection with equivalence classes of $\bigcup_n \widetilde{W}_n$, the inductive limit of braid groups of type $\mathtt{B}$, respect to the equivalence relation $\sim_{\mathtt B}$:
\begin{itemize}
  \item[(i)] $\alpha\beta \sim_{\mathtt B} \beta\alpha $
  \item[(ii)] $\alpha\sim_{\mathtt B} \alpha\sigma_n$ and $\alpha\sim_{\mathtt B} \alpha\sigma_n^{-1}$
\end{itemize}
for all $\alpha, \beta \in W_n$.

\begin{figure}[h!]
\begin{center}
  \includegraphics{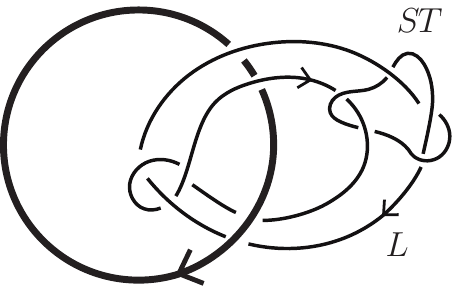}
	\caption{A link in the solid torus.}\label{flink}
	\end{center}
 \end{figure}

We set
\begin{equation}\label{CapitalLambda}
\mathsf{L} := \frac{\z - ({\U}- {\U}^{-1})\x}{\z} \quad \text{and} \quad D:=\frac{1}{z \sqrt{\mathsf{L}}}.
\end{equation}
And, let us denote $\pi_{\mathsf{L}}$ the representation of $\widetilde{W}_n$ in $\E$, given by $\sigma_i \mapsto \sqrt{\mathsf{L}}T_i$ and $\rho_1\mapsto B_1$. Then, for $\alpha\in \widetilde{W}_n $, we define
\begin{equation}\label{inv1}
\overline{\Delta}_{\mathtt{B}}(\alpha):=(D)^{n-1}(\mathtt{tr}_n\circ \pi_{\mathsf{L}})(\alpha)
\end{equation}
it is well know that the previous expression can be rewritten as follows
\begin{equation}\label{inv1}
\overline{\Delta}_{\mathtt{B}}(\alpha)=(D)^{n-1}(\sqrt{\mathsf{L}})^{e(\alpha)}(\mathtt{tr}_n\circ \pi)(\alpha)
\end{equation}
where $e(\alpha)$ is the exponent sum of the $\sigma_i$'s appearing in the braid $\alpha$, and $\pi$ is the natural representation of $\widetilde{W}_n$ in $\E$.

\begin{theorem}
  Let $L$ be a link in $ST$ obtained by closure a braid $\alpha\in \widetilde{W}_n$. Then the map $L\mapsto\overline{\Delta}_{\mathtt{B}}(\alpha) $ defines an isotopy invariant of links in $ST$.
\end{theorem}
\begin{proof}
  The proof follows by using the Markov trace properties, and by the definition of the normalization element $\mathsf{L}$.
\end{proof}

\begin{remark}\rm
  Note that the classical links can be regarded as a link in $ST$, in fact, a classical link can be obtained by closure a braid $\alpha\in \widetilde{W}_n$ whose doesn't contain $\rho_1$ in its expression. Thus, the invariant $\overline{\Delta}_{\mathtt{B}}$ restricted to classical links coincide with the invariant $\overline{\Delta}$ given in \cite[Section 6]{aijuMMJ}, and therefore it's more powerful than the Homflypt polynomial in that case.\\
\end{remark}
\begin{remark}\rm
The Markov trace $\mathtt{tr}$ from Theorem~\ref{Markovtrace} was constructed with the aim to define invariants for \lq\lq tied links in the solid torus\rq\rq, having as reference \cite{aiju2}. However, to do that, it is necessary to introduce these new objects from the beginning, which is a problem itself. Then, we will study this subject in a future work (in progress).
\end{remark}

\bibliography{BT_type_b}{}
\bibliographystyle{acm}

\end{document}